\documentclass[11pt,a4paper]{article}
\usepackage{graphicx}
\usepackage{amsmath}
\usepackage{graphicx}
\usepackage{verbatim}
\usepackage{color}
\usepackage{subfigure}
\usepackage{float}
\usepackage{authblk}
\usepackage[colorlinks]{hyperref}

\usepackage{amsfonts}
\usepackage{textcomp}
\usepackage{eurosym}
\usepackage{enumerate}

\usepackage{textcomp}
\usepackage{eurosym}
\usepackage{array}
\usepackage{amssymb}
\usepackage{subfigure}

\usepackage{amsthm, amsmath}
\theoremstyle{plain}
\theoremstyle{definition}
\newtheorem{thm}{Theorem}
\newtheorem{cor}{Corollary}

\newtheorem{prop}{Proposition}

\newtheorem{lem}{Lemma}
\newtheorem{remark}{Remark}

\begin{document}
\title{The rank of the semigroup of order-, fence-, and parity-preserving partial  injections on a finite set}
\author[1]{Apatsara Sareeto}
\author[2*]{Jörg Koppitz}

\affil[1]{Corresponding author. Institue of Mathematics, University of Potsdam, Potsdam, 14476, Germany  (E-mail: channypooii@gmail.com)}
\affil[2*]{Institute of Mathematics and Informatics, Bulgarian Academy of Sciences, Sofia, 1113, Bulgaria  (E-mail: koppitz@math.bas.bg).}

\maketitle 
\begin{abstract} 
The monoid of all partial injections on a finite set (the symmetric inverse semigroup) is of particular interest because of the well-known Wagner-Preston Theorem. Let $n$ be a positive natural number and $PFI_n$ be
the semigroup of all fence-preserving partial one-one maps of $\{1,...,n\}$ into itself
with respect to composition of maps and the fence $1 \prec 2 \succ 3 \prec\cdot\cdot\cdot \ n$. There is considered the inverse semigroup $IOF_n^{par}$ of all $\alpha\in PFI_n$ such that $\alpha$ is regular in $PFI_n$, order-preserving with respect to the order $1 < 2 < \cdots < n$ and parity-preserving. According to the main result of the paper, it is $3n-6$ the least of the cardinalities of the generating sets of $IOF_n^{par}$ for $4 \leq n$. There is determined a concrete representation of a generating set of minimal size.



\end{abstract}

\section{Introduction and Preliminaries}\label{sec1}
It is well-known that every finite semigroup is isomorphic to a subsemigroup of a suitable finite transformation semigroup. It is the analog of Cayley's Theorem for finite groups. Hence, the transformation semigroups and their subsemigroups have an important role in semigroup theory, as the symmetric groups in group theory. In inverse semigroup theory, the Wagner-Preston Theorem states that every inverse semigroup is isomorphic to a subsemigroup of a suitable symmetric inverse semigroup. \\ 
\indent Let $\overline{n}$ be a finite set with $n$ elements ($n$ is a positive integer), say $\overline{n}=\{1,...,n\}$. We denote by $PT_n$ the monoid (under composition) of all partial transformations on $\overline{n}$. A partial injection $\alpha$ on the set $\overline{n}$ is a one-to-one function from a subset $A$ of $\overline{n}$ into $\overline{n}$. The set of all partial denoted by $I_n$. The domain of $\alpha$ is the set $A$, denoted by $dom(\alpha)$. The range of $\alpha$ is denoted by $im(\alpha)$. The empty transformation will be denoted by $\varepsilon$, it is the transformation with  $dom(\varepsilon)=\emptyset$. The set $I_n$ (under composition) forms a monoid, which is called a symmetric inverse semigroup. The symmetric inverse semigroup was introduced by Wagner \cite{wagner}. If the domain has cardinality $m$, which is also the cardinality of the range, then the transformation $\alpha$ is said to be of rank $m$, in symbol: rank$(\alpha)= m$. \\
\indent Let $S$ be a semigroup, and let $A$ be a non-empty subset of $S$. Then the subsemigroup generated by $A$, that is the smallest subsemigroup of $S$ containing $A$, is denoted by $\langle A \rangle$. If a semigroup $S$ has a finite subset $A$ such that $S=\langle A \rangle $, then $S$ is called a finitely generated semigroup. The rank of a finitely generated semigroup $S$ is defined by rank$(S)=min\{\lvert A\rvert:\langle A \rangle=S\}$.  A generating set for $S$ is called a minimal generating set if no proper subset of it generates $S$. \\
\indent  Now, we  consider a linear order $1<2<\cdot\cdot\cdot<n$ on $\overline{n}$. We say that a transformation $\alpha\in PT_n$ is order-preserving if $x< y$ implies $x\alpha\leq y\alpha$, for all $x,y\in dom(\alpha)$. We denote by $PO_n$ the submonoid of $PT_n$ of all order-preserving partial transformations and by $POI_n$ the monoid $PO_{n}\cap I_n$ of all order-preserving partial injections on $\overline{n}$. \\
\indent Ganyuskin and Mazorchuk \cite{Gany} described the maximal subsemigroups of the semigroup $POI_n$. In \cite{Ilinka}, Dimitrova and Koppitz characterized the maximal subsemigroups of the ideals of the semigroup $POI_n$. Fernandes calculated the size of $POI_n$ in \cite{fer}, it has the size $ {2n \choose n} $. Moreover, Fernandes has found that $POI_n$ is generated by $J_{n-1}$, whenever $J_k$ is the $J$-class of $POI_n$ consisting of the maps in $POI_n$ of rank $k$, for all $0\leq k\leq n$. Notice that $J_0=\{\varepsilon\}$ and $J_n=\{id_{\overline{n}}\}$, where $id_{\overline{n}}$ is the identity mapping on $\overline{n}$. Recently, Annis and Lopez \cite{Ann} have shown that  $POI_n$ has $(n-1)!$ minimal generating sets. \\
\indent The rank of the monoid $PO_n$ was established by Gomes and Howie \cite{Gomes}  and Garba \cite{Garba} studied the idempotent ranks of certain semigroups of order-preserving transformations in 1994. Later in 2001, Fernandes calculated that the rank of the monoid $POI_n$ is $n$. \\
\indent A non-linear order that is close to a linear order in some sense is the so-called zig-zag order. The pair $(\overline{n}, \preceq)$ is called a zig-zag poset or fence if 
\begin{center}
	$1 \prec 2 \succ \cdot\cdot\cdot \prec n - 1 \succ n$ or $1 \succ 2 \prec \cdot\cdot\cdot \succ n - 1 \prec n$ if n is odd \\
	and  $1 \prec 2 \succ \cdot\cdot\cdot \succ n - 1 \prec n$ or $1 \succ 2 \prec \cdot\cdot\cdot \prec n - 1 \succ n$ if n is even.
\end{center}
The definition of the partial order $\preceq$ is self-explanatory. Transformations on fences were first considered by Currie and Visentin \cite{currie} as well as Rutkowski \cite{Rutkowski}. We observe that every element in a fence is either minimal or maximal. Without loss of generality, let $1 \prec 2 \succ 3 \prec\cdot\cdot\cdot \succ n$ and $1 \prec 2 \succ 3 \prec\cdot\cdot\cdot \succ n-1\prec n$, respectively. Such fences are also called up-fences. The fence $1 \succ 2 \prec 3 \succ \cdot\cdot\cdot \prec n$ and $1 \succ 2 \prec 3 \succ \cdot\cdot\cdot \prec n-1\succ n$, respectively, would be called down-fence. We avoid both notations up-fence and down-fence. To check whether a fence is an up-fence or down-fence, we need that $1$ and $2$ are comparable for $\preceq$. Recall that $x, y \in \overline{n}$ are comparable with respect to $\preceq$ if $x \prec y$ or $x = y$ or $x \succ y$. Otherwise, $x$ and $y$ are called incomparable. But the restriction that $1$ and $2$ belong to the fence and are comparable is an unnecessary restriction for the concept fence since instead of $\overline{n}$ one could choose another $n$-element set or one could define $\preceq$ on $\overline{n}$ such that $1$ and $2$ are incomparable. But if the fence $(\overline{n}, \preceq)$ is defined as above (which is the most natural way) then we observe that any $x, y \in \overline{n}$ are comparable if and only if $x \in \{y - 1, y, y + 1\}$. \\ \indent We say that a transformation $\alpha \in I_n$ is fence-preserving if $x \prec y$ implies that $x\alpha \prec y\alpha$, for all $x, y \in dom (\alpha)$. We denote by $PFI_n$ the submonoid of $I_n$ of all  fence-preserving partial injections of $\overline{n}$. Fernandes et al. characterized the full transformations on $\overline{n}$ preserving the zig-zag order \cite{fer 2}. It is worth mentioning that several other properties of monoids of fence-preserving full transformations were also studied. In \cite{Jit, Tanya}, Srithus et al., the regular elements of these monoids were discussed. Some relative ranks of the monoid of all partial transformations preserving an infinite zig-zag order were determined in \cite{Dimi}. We denote by $IF_n$ the inverse subsemigroup of all regular elements in $PFI_n$. It is easy to see that $IF_n$ is the set of all $\alpha\in PFI_n$ with $\alpha^{-1}\in PFI_n$. For the case that $n$ is even, it is proved that rank$(IF_n) = n+1$ and a concrete generating set of $IF_n$ with $n+1$ elements is given in \cite{Ilinka 2}. Later in 2021, for the case that $n$ is odd, Koppitz and Musunthia \cite{Kopp} calculated that the rank of $IF_n$ is $5$ or $\frac{n-5}{2}+\lfloor{\frac{n+6}{4}}\rfloor \lfloor{\frac{n+7}{4}}\rfloor$ whenever $n=3$ and $n\geq 5$ is odd, respectively. Fence-preserving transformations are also studied in \cite{fer 2, Jen, Loh, Sir}. For general background on semigroups and standard notations, we refer the reader to \cite{clif, Howie}. \\ \indent The previous facts have given us the main inspiration for the study of a submonoid of $POI_n \bigcap IF_n$, namely the monoid $IOF_n^{par}$ of all $\alpha\in POI_n \bigcap IF_n$. In the present paper, we restrict us to the case that $x$ and $x\alpha$ have the same parity for all $x\in dom (\alpha)$. It is easy to verify that $IOF_n^{par}$ forms a monoid, the inverse partial injection exists for any  $\alpha\in IOF_n^{par}$  and is order-, fence-, and parity-preserving. This implies that $IOF_n^{par}$ is an inverse submonoid of $I_n$. We focus our attention on generating sets of $IOF_n^{par}$. For $n\in\{1,2,3\}$, we observe that $IOF_1^{par}= \{ \bigl(\begin{smallmatrix}
1\\1 \end{smallmatrix}\bigr), \varepsilon\}$, $ IOF_2^{par}=\{\bigl(\begin{smallmatrix}
1 &2\\
1 &2 \end{smallmatrix}\bigr), \bigl(\begin{smallmatrix}
1\\1 \end{smallmatrix}\bigr), \bigl(\begin{smallmatrix}
2\\2 \end{smallmatrix}\bigr), \varepsilon\}$ and $IOF_3^{par}=\{\bigl(\begin{smallmatrix}
1 &2&3\\
1 &2&3 \end{smallmatrix}\bigr),\bigl(\begin{smallmatrix}
1 \\1 \end{smallmatrix}\bigr), \bigl(\begin{smallmatrix}
2\\2 \end{smallmatrix}\bigr), \bigl(\begin{smallmatrix}
3\\3 \end{smallmatrix}\bigr),\bigl(\begin{smallmatrix}
1 &2\\
1 &2 \end{smallmatrix}\bigr), \bigl(\begin{smallmatrix}
1 &3\\
1 &3 \end{smallmatrix}\bigr),\bigl(\begin{smallmatrix}
2 & 3\\
2 & 3 \end{smallmatrix}\bigr), \bigl(\begin{smallmatrix}
1\\3 \end{smallmatrix}\bigr), \bigl(\begin{smallmatrix}
3\\1 \end{smallmatrix}\bigr),  \varepsilon\}$. It is routine to calculate the ranks of these three monoids. We obtain rank$(IOF_1^{par}) = 1$,  rank$(IOF_2^{par}) = 2$, and  rank$(IOF_3^{par}) = 5$. Let $n \geq 4$ for the  rest of this paper.
We can characterize the transformations in $IOF_n^{par}$ as follows:
\begin{prop} \label{4 choice}
	Let  $\alpha = \bigl(\begin{smallmatrix}
	d_1 &<&d_2&<&   \cdots & < &d_p \\
	m_1 & &m_2 & & \cdots &    & m_p
	\end{smallmatrix}\bigr) \in I_n$. Then $\alpha\in IOF_n^{par}$ if and only if  the following four conditions hold.\\
	(i) $m_1<m_2<\cdot\cdot\cdot<m_p$. \\
	(ii) $d_1$ and $m_1$ have the same parity. \\
	(iii) $d_{i+1}-d_i=1$ if and only if $m_{i+1}-m_i=1$ for all $i\in\{1,...,p-1\}$.  \\
	(iv) $d_{i+1}-d_i$ is even if and only if $m_{i+1}-m_i$ is even for all $i\in\{1,...,p-1\}$.
\end{prop}
\begin{proof}   $(\Rightarrow)$: (i) and (ii) hold since $\alpha$ is order- and parity-preserving, respectively. 
	(iii): Since $\alpha\in IF_n$, we have $d_{i+1}\alpha-d_i\alpha=1$, i.e. $m_{i+1}-m_i=1$, if and only if $d_{i+1}-d_i=1$, for all $i\in\{1,...,p-1\}$. 	(iv): Suppose  $d_{i+1}-d_i$ is even. Then $d_{i+1}$ and $d_i$ have the same parity. Moreover, $\alpha$ is parity-preserving. This implies $d_{i+1}\alpha$ and $d_i\alpha$ have the  same parity, i.e. $m_{i+1}-m_i$ is even. The converse direction can be proved dually. \\	
	
	\indent \quad $(\Leftarrow)$: By (i), we get $\alpha$ is order-preserving. Let $i\in\{1,...,p-1\}$ and suppose $d_i$ and $m_i$ have the same parity. Then $d_i-m_i=2k$ for some integer $k$. By (iv), we have $(d_{i+1}-d_i)-(m_{i+1}-m_i)=2l$ for some integer $l$. We obtain $2l=d_{i+1}-m_{i+1}-(d_i-m_i)=d_{i+1}-m_{i+1}-2k$, i.e. $d_{i+1}-m_{i+1}=2(l+k)$. This implies $d_{i+1}$ and $m_{i+1}$ have the same parity. Together with (ii), we can conclude that $\alpha$ is parity-preserving. Let  $x\prec y$. This provides $\lvert x- y \rvert=1$. We have $ \lvert x\alpha - y\alpha \rvert = 1 $ by (iii). Since $\alpha$ is parity-preserving, $ \lvert x\alpha - y\alpha \rvert = 1 $ and $x\prec y$ give $x\alpha \prec y\alpha$. So, $\alpha\in PFI_n$. Similarly, we can show that $\alpha^{-1}\in PFI_n$, i.e. $\alpha\in IF_n$. Therefore, $\alpha\in IOF_n^{par}$.
\end{proof}
\indent	Let $\overline{v}_i$ be the partial identity with the domain $\overline{n}\backslash \{i\}$ for all $i\in\{1,...,n\}$ and  let $Id_{\overline{n}}$ be the set of all partial identities. Further, let 	

\begin{center}   
	$\overline{u}_i = \begin{pmatrix}
	1 &  \cdots  & i & i+1 & i+2 & i+3 & i+4 &  \cdots   & n \\
	3 &   \cdots   & i+2 & - & -& - & i+4 &  \cdots     & n 
	\end{pmatrix}$ \end{center}  and  $\overline{x}_i=(\overline{u}_i)^{-1}$ for all $i\in\{1,...,n-2\}$. By Proposition \ref{4 choice}, it is easy to verify that $\overline{u}_i$ as well as $\overline{x}_i$, $i\in\{1,...,n-2\}$, belong to $IOF_n^{par}$. We will show that
\begin{center}
	$A_n=\{\overline{v}_1,...,\overline{v}_n , \overline{u}_1,...,\overline{u}_{n-4},\overline{u}_{n-2},\overline{x}_1,...,\overline{x}_{n-4} ,\overline{x}_{n-2} \}$ \end{center} is a generating set of minimal size for the monoid $IOF_{n}^{par}$. Clearly, $\langle A_n \rangle\subseteq IOF_{n}^{par}$. It is easy to see that all partial identities including the empty transformation are generated by $\{\overline{v}_1,...,\overline{v}_n\}\subseteq A_n$.
In the next section, we will present any $\alpha\in IOF_n^{par}\backslash (Id_{\overline{n}}\cup \{\varepsilon\})$ as a normal form in $\langle A_n \rangle$.
\section{Normal forms of transformations in $\langle A_n \rangle$}  
In this section, we will find a generating set for $IOF_n^{par}$. We will fix now an $\alpha \in IOF_n^{par}\backslash (Id_{\overline{n}}\cup \{\varepsilon\})$, say  $\alpha = \bigl(\begin{smallmatrix}
d_1 & < & \cdots & < & d_p \\
m_1 &  &  \cdots &    & m_p
\end{smallmatrix}\bigr)$, where $p=\lvert\mbox{rank}(\alpha) \rvert$. 
In order to show that $\alpha\in\langle A_n \rangle$,  we consider a word $w_\alpha$ over the alphabet $X_n=\{v_1,...,v_n,u_1,...,u_{n-2},\\ x_1,...,x_{n-2}\}$  and show that $\overline{w}_\alpha=\alpha$, where $\overline{w}_\alpha$ is the transformation that we obtain from the word $w_\alpha$ by replacing any letter $a$ in $w_\alpha$ by the transformation $\overline{a}$. For a word $z=a_1...a_k$ over $X_n$,  let $z^{-1}=a_ka_{k-1}...a_1$.  Let $x_{i,j}=x_ix_{i+2}...x_{i+2j-2}$ and  $u_{i,j}=u_iu_{i+2}...u_{i+2j-2}$ for $i\in\{1,...,n-2\}, j\in\{1,...,\lfloor{\frac{n-i}{2}}\rfloor\}$. Further, let $W_x=\{x_{i,j} : i\in\{1,...,n-2\}, j\in \{1,...,\lfloor{\frac{n-i}{2}}\rfloor\}$
and $W_u=\{u_{i,j} : i\in\{1,...,n-2\},j\in \{1,...,\lfloor{\frac{n-i}{2}}\rfloor\}$.  For a non-empty set $A=\{a_1<\cdots<a_r\}\subseteq \overline{n}$, for some $r\in\overline{n}$, we put 
$v_A=v_{a_1}...v_{a_r}$. Additional, $v_\emptyset$ is the empty word $\epsilon$.  First, we construct the word $w_\alpha$.  
\\   
\indent There are a unique $l\in\{0,1,...,p-1\}$ and a unique set $\{r_1,...,r_l\}\subseteq \{1,...,p-1\}$ such that (i)-(iii) are satisfied:\\
(i) $r_1<\cdot\cdot\cdot<r_l$;\\
(ii) $d_{r_i+1}-d_{r_i}\neq m_{r_i+1}-m_{r_i}$ for $i\in\{1,...,l\}$; \\
(iii) $d_{i+1}-d_{i}= m_{i+1}-m_{i}$ for $i\in\{1,...,p-1\}\backslash\{r_1,...,r_l\}$.  \\ 
Note that $l=0$ means $\{r_1,...,r_l\}=\emptyset$. Further, we put $r_{l+1} = p$.
For $i\in\{1,...,l\}$, we define
\begin{align*}
w_i 
=& \begin{cases}
x_{m_{r_i},\frac{(m_{r_i+1}-m_{r_i})-(d_{r_i+1}-d_{r_i})}{2}} \ \mbox{if} \ m_{r_i+1}-m_{r_i}>d_{r_i+1}-d_{r_i}; \\
u_{d_{r_i},\frac{(d_{r_i+1}-d_{r_i})-(m_{r_i+1}-m_{r_i})}{2}} \ \mbox{if} \ m_{r_i+1}-m_{r_i}<d_{r_i+1}-d_{r_i}. 
\end{cases}
\end{align*} Obviously, we have $w_i\in W_x\cup W_u$ for all $i\in\{1,...,l\}$. If $m_p=d_p$ then we put $w_{l+1}=\epsilon$. If $m_p \neq d_p$, we define additionally
\begin{align*}
w_{l+1} 
=& \begin{cases}
x_{m_{p},\frac{d_p-m_p}{2}} \ \mbox{if} \ d_p>m_p ; \\
u_{d_{p},\frac{m_p-d_p}{2}} \ \mbox{if} \ d_p<m_p. 
\end{cases}&&\qedhere
\end{align*} 
Clearly, $w_{l+1}\in W_x\cup W_u$. We will use the notation $w_{k}=u_{i_{k},j_{k}}$ if $w_k\in W_u$ and $w_k=x_{i_k,j_k}$ if $w_k\in W_x$ for all $k\in\{1,...,l+1\}$. For $k\in\{1,...,l+1\}$, we define integers $k_u$ and $k_x$, recursively.    \\    
If $m_p=d_p$ then we put $l+1)_u=(l+1)_x=d_p$. \\
If $w_{l+1}\in W_u$ then we put $(l+1)_u=i_{l+1}$ and $(l+1)_x=i_{l+1}+2j_{l+1}$.  \\
If $w_{l+1}\in W_x$ then we put $(l+1)_u=i_{l+1}+2j_{l+1}$ and $(l+1)_x=i_{l+1}$. \\
Let $k\in\{1,...,l\}$. \\
If $w_k\in W_u$ then we put  $k_u=i_k$ and $k_x= (k+1)_x-a_k-2$ with $a_k=(k+1)_u-i_k-2j_k-2$. \\
If $w_k\in W_x$ then we put $k_u= (k+1)_u-b_k-2 $ and $k_x= i_k $ with $b_k=(k+1)_x-i_k-2j_k-2$.\\

We observe that $1_u,...,(l+1)_u$ and $1_x,...,(l+1)_x$ correspond to the domain of $\alpha$ and the image of $\alpha$, respectively.
\begin{lem} \label{lem zu=dz}  For  all $k\in\{1,...,l+1\}$, we have $k_u=d_{r_k}, k_x=m_{r_k}$.
\end{lem}
\begin{proof} We prove by induction on $k$. First, we prove that  $(l+1)_x=m_{r_{l+1}}$ and $(l+1)_u=d_{r_{l+1}}$. Suppose $m_p\neq d_p$. If $w_{l+1}\in W_x$ then $(l+1)_x=i_{l+1}=m_{r_{l+1}}$ and $(l+1)_u=i_{l+1}+2j_{l+1}=i_{l+1}+2(\frac{d_{r_{l+1}}-m_{r_{l+1}}}{2})=m_{r_{l+1}}+2(\frac{d_{r_{l+1}}-m_{r_{l+1}}}{2})=d_{r_{l+1}}$. 	For $w_{l+1}\in W_u$, the proof is similar. If $m_p=d_p$ then $(l+1)_u=d_p=d_{r_{l+1}}$ and $(l+1)_x=d_p=m_p=m_{r_{l+1}}$. \\
	\indent 	
	Suppose that $(k+1)_u=d_{r_{k+1}}$ and $(k+1)_x=m_{r_{k+1}}$ for some $k\in\{1,...,l\}$. 
	If $w_k=u_{i_k,j_k}\in W_u$ then $k_u=i_k=d_{r_k}$. On the other hand, we have $(d_{r_k+1}-d_{r_k})-(m_{r_k+1}-m_{r_k})= (d_{r_{k+1}}-d_{r_k})-(m_{r_{k+1}}-m_{r_k})$. Then
	$k_x = (k+1)_x-a_k-2 = (k+1)_x-(k+1)_u+i_k+2j_k = m_{r_{k+1}}-d_{r_{k+1}}+d_{r_k}+2(\frac{(d_{r_{k+1}}-d_{r_k})-(m_{r_{k+1}}-m_{r_k})}{2}) = m_{r_k}$.
	For $w_k\in W_x$, the proof is similar.
\end{proof}	
\noindent We consider the word 
\begin{center}
	$w=w_1...w_{l+1}$.
\end{center} From this word, we construct a new word $w_\alpha^*$ by arranging the subwords $x$ belonging $W_x$ in reverse order at the end of the word $w$, replacing $x$ by $x^{-1}$. In other words, we consider the word \begin{center}
	$w_\alpha^*= w_{s_1}...w_{s_a}w_{s_{a+1}}^{-1}...w_{s_{a+b}}^{-1}$
\end{center} such that $w_{s_1},...,w_{s_a} \in W_u$, $w_{s_{a+1}},...,w_{s_{a+b}}\in W_x$ and $\{w_{s_1},...,w_{s_a},w_{s_{a+1}},...,w_{s_{a+b}}\}=\{w_1,...,w_{a+b}\}$, where $s_1<\cdot\cdot\cdot<s_a, s_{a+b}<\cdot\cdot\cdot<s_{a+1}$ and $a,b \in \overline{n}\cup\{0\}$ with
\begin{center}
	$a+b=
	\begin{cases}
	l  & \mbox{if} \ d_p=m_p;   \\
	l+1 & \mbox{if} \ d_p\neq m_p.
	\end{cases}$
\end{center} For convenient, $a=0$ means $w_\alpha^*=w_{s_{a+1}}^{-1}...w_{s_{a+b}}^{-1}  $ and $b=0$ means $w_\alpha^*=w_{s_1}...w_{s_a}$.  Now we add recursively letters from the set $\{v_1,...,v_n\}\subseteq X_n$ to the word $w_\alpha^*$, obtaining new words $\lambda_0,\lambda_1,...,\lambda_p$. \\

\noindent (1)	For $d_p \leq n-2$: \\
\indent	(1.1) if $m_p<d_p$ then $\lambda_0= v_{d_p+2}...v_nw^*_\alpha$; \\
\indent	(1.2) if $n-1>m_p>d_p$ then $\lambda_0= v_{m_p+2}...v_nw^*_\alpha$; \\ 
\indent	(1.3) if $m_p=d_p$ then $\lambda_0= v_{m_p+1}...v_nw^*_\alpha$; \\
\indent otherwise $\lambda_0=w^*_\alpha$. \\
(2)	 If  $d_p=m_p=n-1$ then $\lambda_0=v_nw^*_\alpha$.  Otherwise $\lambda_0=w^*_\alpha$.  \\
(3)	 For $k\in\{2,...,p\}$: \\
\indent(3.1) if $2\leq m_k-m_{k-1}=d_k-d_{k-1}$ then $\lambda_{p-k+1}= v_{d_{k-1}+1}...v_{d_k-1}\lambda_{p-k}$; \\
\indent(3.2) if $2< m_k-m_{k-1}<d_k-d_{k-1}$ then \\ \indent $\lambda_{p-k+1} = v_{d_k-(m_k-m_{k-1}-2)}...v_{d_k-1}\lambda_{p-k}$; \\
\indent	(3.3)	if $m_k-m_{k-1}>d_k-d_{k-1}>2$ then $\lambda_{p-k+1}=v_{d_{k-1}+2}...v_{d_k-1}\lambda_p$; \\
\indent otherwise $\lambda_{p-k+1}=\lambda_{p-k}$. \\
(4)	If $d_1=1$ or $m_1=1$ then $\lambda_p=\lambda_{p-1}$. \\
(5)	 If $1<d_1\leq m_1$ then $\lambda_p=v_1...v_{d_1-1}\lambda_{p-1}$. \\
(6)	 If $1<m_1<d_1$ then $\lambda_p=v_{d_1-m_1+1}...v_{d_1-1}\lambda_{p-1}$. \\

\indent The word $\lambda_p$ induces a set  $A=\{a\in\overline{n}: v_a\in var(\lambda_p)\}$ and it is easy to verify that  $\rho\notin A$ for all $\rho\in dom(\alpha)$. We put $w_\alpha=\lambda_p$. The word $w_\alpha$ has the form $w_\alpha=v_Aw_\alpha^*$. Clearly, the word $w_\alpha$ defines a transformation $\overline{w}_\alpha$. Moreover, we will point out that $\{\overline{w_\beta^*} : \beta\in IOF_n^{par}\}$ provides a set of normal forms for the products in $\langle A_n \rangle$. To verify  this, it is enough to show that  $\alpha=\overline{w}_\alpha$ since $\alpha$ is fixed but arbitrary. In order to prove the equality of these both transformations, we verify that $\rho\alpha=\rho\overline{w}_\alpha$ for all $\rho\in dom(\alpha)$  and  $\rho\notin dom(\overline{w}_\alpha)$, whenever $\rho\notin dom(\alpha)$. The following lemma will show that the words $w_{s_1},...,w_{s_a}$ as well as the words $w_{s_{a+1}},...,w_{s_{a+b}}$ have pairwise no common variables.
\begin{lem} \label{1-4} Let $k<k'\leq a+b$. \\
	\indent  (i) If $w_k\in W_x$ and $w_{k'}\in W_x$ then $i_k+2j_k+1<i_{k'}$. \\
	\indent (ii) If $w_k\in W_u$ and $w_{k'}\in W_u$ then $i_k+2j_k+1<i_{k'}$. \\
	\indent (iii) If $w_k\in W_u$ then $i_k+2j_k+2\leq(k+1)_u$. \\
	\indent (iv) If $w_k\in W_x$ then $i_k+2j_k+2\leq(k+1)_x$.
\end{lem}
\begin{proof}
	(i) We have $k_x=i_k$ and ${k'}_x=i_{k'}$. Clearly, $k<{k'}$ implies $m_{r_k+1}\leq m_{r_{k'}}$. Moreover, we have $d_{r_k+1}-d_{r_k}>1$ by definition of $r_k$. This implies $m_{r_k+1}-(d_{r_k+1}-d_{r_k})+1 < m_{r_k+1}$,  
	$m_{r_k}+2(\frac{(m_{r_{k+1}}-m_{r_k})-(d_{r_{k+1}}-d_{r_k})}{2})+1< m_{r_k+1}\leq m_{r_k'}$, $m_{r_k}+2j_k+1 < m_{r_{k'}}$, and thus 
	$i_k+2j_k+1 < i_{k'}$  by Lemma \ref{lem zu=dz}.   \\
	
	\noindent	(ii) The proof is similar to (i). \\
	
	\noindent	(iii) We have that $d_{r_k+1}-d_{r_k}>m_{r_k+1}-m_{r_k}> 1$ and $d_{r_k+1}\leq d_{r_{k+1}}$. This implies $d_{r_k+1}+1-(m_{r_k+1}-m_{r_k})<d_{r_k+1}$ and 
	$d_{r_k+1}-(m_{r_k+1}-m_{r_k})+2\leq d_{r_k+1}\leq d_{r_{{k+1}}}$. 
	Moreover, we have $d_{r_k+1}-(m_{r_k+1}-m_{r_k})+2 = d_{r_k}+2(\frac{(d_{r_{k+1}}-d_{r_k})-(m_{r_{k+1}}-m_{r_k})}{2})+2= d_{r_k}+2j_k+2 = i_k+2j_k+2$ and $d_{r_k}=i_k$. Therefore, $i_k+2j_k+2\leq d_{r_k+1}\leq d_{r_{k+1}}=(k+1)_u$ by Lemma \ref{lem zu=dz}. \\
	
	\noindent	(iv) The proof is similar to (iii).
\end{proof}
We observe that $d_{r_k}$ can be calculated from $m_{r_k}$(and conversely) using the length of appropriate subwords of $w_\alpha^*$. Subsequently, we will use  that $\lvert w_{s_k} \rvert = j_{s_k}$ for all $k\in\{1,...,a\}$ and $\lvert w_{s_{a+d}}^{-1} \rvert = j_{s_{a+d}}$ for all $d\in\{1,...,b\}$.
\begin{lem} \label{u>x} Let $d\in\{1,...,b\}$. If there is the least $k\in\{1,...,a\}$ such that $s_k>s_{a+d}$  then $d_{r_{s_{a+d}}}=i_{s_{a+d}}+2\lvert w_{s_{a+d}}^{-1}...w_{s_{a+1}}^{-1}\rvert - 2\lvert w_{s_k}...w_{s_a}\rvert.$ Otherwise, $d_{r_{s_{a+d}}}=i_{s_{a+d}}+2\lvert w_{s_{a+d}}^{-1}...w_{s_{a+1}}^{-1}\rvert.$
\end{lem}
\begin{proof} Suppose there is the least $k\in\{1,...,a\}$ such that $s_k>s_{a+d}$. 
	Note that $i_{s_{a+d}}+2\lvert w_{s_{a+d}}^{-1}...w_{s_{a+1}}^{-1}\rvert-2\lvert w_{s_k}...w_{s_a}\rvert=i_{s_{a+d}}+2\lvert w_{s_{a+d}}^{-1}\rvert+\cdot\cdot\cdot+2\lvert w_{s_{a+1}}^{-1}\rvert-2\lvert w_{s_k}\rvert-\cdot\cdot\cdot-2\lvert w_{s_a}\rvert$ 
	and $\lvert w_{s_{a+d}}^{-1}\rvert= 
	\frac{(m_{r_{s_{a+d}+1}}-m_{r_{s_{a+d}}})-(d_{r_{s_{a+d}}+1}-d_{r_{s_{a+d}}})}{2}$ 
	since $m_{r_{{s_{a+d}}+1}}-m_{r_{s_{a+d}}+1}=d_{r_{{s_{a+d}}+1}}-d_{r_{s_{a+d}}+1}$. 
	Let $t\in\{{s_{a+d}}+1,...,l\}$. If $w_{t}\in W_x$ then $\lvert w_{t}\rvert=\frac{(m_{r_{t+1}}-m_{r_{t}})-(d_{r_{t+1}}-d_{r_{t}})}{2}$.	If $w_{t}\in W_u$   then we have  $-\lvert w_{t}\rvert=-(\frac{(d_{r_{t+1}}-d_{r_{t}})-(m_{r_{t+1}}-m_{r_{t}})}{2})$. Moreover $\lvert w_{l+1}\rvert=\frac{m_{r_{l+1}}-d_{r_{l+1}}}{2}$, whenever $w_{l+1}\in W_x$ and $-\lvert w_{l+1}\rvert=-(\frac{d_{r_{l+1}}-m_{r_{l+1}}}{2})$,  whenever $w_{l+1}\in W_u$.  We observe that $s_k>s_{a+d}$ provides
	$i_{s_{a+d}}+2\lvert w_{s_{a+d}}^{-1}\rvert+\cdot\cdot\cdot+2\lvert w_{s_{a+1}}^{-1}\rvert-2\lvert w_{s_k}\rvert-\cdot\cdot\cdot-2\lvert w_{s_a}\rvert= i_{s_{a+d}}+(m_{r_{{s_{a+d}}+1}}-m_{r_{s_{a+d}}}-d_{r_{{s_{a+d}}+1}}+d_{r_{s_{a+d}}})+  (m_{r_{{s_{a+d}}+2}}-m_{r_{{s_{a+d}}+1}}-d_{r_{{s_{a+d}}+2}}+d_{r_{{s_{a+d}}+1}})+\cdot\cdot\cdot+ (m_{r_{l+1}}-m_{r_{l}}-d_{r_{l+1}}+d_{r_{l}})+(d_{r_{l+1}}-m_{r_{l+1}}) = i_{s_{a+d}}-m_{r_{s_{a+d}}}+d_{r_{s_{a+d}}} = i_{s_{a+d}}-i_{s_{a+d}}+d_{r_{s_{a+d}}}= d_{r_{s_{a+d}}}$. 
	\\	\indent Suppose now that $s_k<s_{a+d}$  for all $k\in\{1,...,a\}$ or $w_\alpha^*=w_{s_{a+1}}^{-1}...w_{s_{a+b}}^{-1}$. Then $i_{s_{a+d}}+2\lvert w_{s_{a+d}}^{-1}...w_{s_{a+1}}^{-1}\rvert=i_{s_{a+d}}+2\lvert w_{s_{a+d}}^{-1}\rvert+\cdot\cdot\cdot+2\lvert w_{s_{a+1}}^{-1}\rvert= i_{s_{a+d}}+(m_{r_{{s_{a+d}}+1}}-m_{r_{s_{a+d}}}-d_{r_{{s_{a+d}}+1}}+d_{r_{s_{a+d}}})+  (m_{r_{{s_{a+d}}+2}}-m_{r_{{s_{a+d}}+1}}-d_{r_{{s_{a+d}}+2}}+d_{r_{{s_{a+d}}+1}})+\cdot\cdot\cdot+ (m_{r_{l+1}}-m_{r_{l}}-d_{r_{l+1}}+d_{r_{l}})+(d_{r_{l+1}}-m_{r_{l+1}}) 
	= i_{s_{a+d}}-m_{r_{s_{a+d}}}+d_{r_{s_{a+d}}} 
	= i_{s_{a+d}}-i_{s_{a+d}}+d_{r_{s_{a+d}}} 
	= d_{r_{s_{a+d}}}.$
\end{proof}
Similarly, we can prove:
\begin{lem} \label{x>u}
	Let $k\in\{1,...,a\}$. If there is the greatest $d\in\{1,...,b\}$ such that $s_{a+d}>s_{k}$ then $m_{r_{s_{k}}}=i_{s_k}+2\lvert w_{s_k}...w_{s_a}\rvert-2\lvert w_{s_{a+d}}^{-1}...w_{s_{a+1}}^{-1}\rvert.$ Otherwise, $m_{r_{s_{k}}}=i_{s_k}+2\lvert w_{s_k}...w_{s_a}\rvert.$
\end{lem}
Moreover, we have two technical lemmas.
\begin{lem} \label{A tran= word}
	Let  $k\in\{1,...,a\}$. It holds $d_{r_{s_k+1}}-(m_{r_{s_k+1}}-m_{r_{s_k}}-2)=i_{s_k}+2\lvert w_{s_k}\rvert+2$. 
\end{lem}
\begin{proof}  First, we consider the case there is the greatest $d\in\{1,...,b\}$ such that $s_{a+d}>s_k$. Suppose  $w_{s_{k}+1}\in W_u$, i.e. $s_k+1=s_{k+1}$. Then by Lemmas \ref{lem zu=dz} and \ref{x>u}, we obtain
	$d_{r_{{s_k}+1}}=i_{s_{k+1}}, m_{r_{{s_k}+1}}=i_{s_{k+1}}+2\lvert w_{s_{k+1}}...w_{s_a}\rvert-2\lvert w_{s_{a+d}}^{-1}...w_{s_{a+1}}^{-1}\rvert$, and $m_{r_{s_k}}=i_{s_{k}}+2\lvert w_{s_{k}}...w_{s_a}\rvert-2\lvert w_{s_{a+d}}^{-1}...w_{s_{a+1}}^{-1}\rvert$. This implies 
	$ d_{r_{{s_k}+1}}-(m_{r_{{s_k}+1}}-m_{r_{s_k}}-2) = i_{s_{k+1}}-i_{s_{k+1}}-2\lvert w_{s_{k+1}}...w_{s_a}\rvert+2\lvert w_{s_{a+d}}^{-1}...w_{s_{a+1}}^{-1}\rvert+i_{s_k}+2\lvert w_{s_{k}}...w_{s_a}\rvert-\lvert w_{s_{a+d}}^{-1}...w_{s_{a+1}}^{-1}\rvert+2 
	= i_{s_k}+2\lvert w_{s_{k}}\rvert+2$. Suppose  $w_{s_{k}+1}\in W_x$, i.e. $s_k+1=s_{a+d}$. Then by Lemmas \ref{lem zu=dz}, \ref{u>x}, and, \ref{x>u}, we get 
	$d_{r_{{s_k}+1}}=i_{s_{a+d}}+2\lvert w_{s_{a+d}}^{-1}...w_{s_{a+1}}^{-1}\rvert-2\lvert w_{s_{k+1}}...w_{s_a}\rvert, \ m_{r_{{s_k}+1}}=i_{s_{a+d}}$, and $m_{r_{s_k}}=i_{s_{k}}+2\lvert w_{s_{k}}...w_{s_a}\rvert-2\lvert w_{s_{a+d}}^{-1}...w_{s_{a+1}}^{-1}\rvert$. This implies  	$ d_{r_{{s_k}+1}}-(m_{r_{{s_k}+1}}-m_{r_{s_k}}-2) 
	=i_{s_{a+d}}+2\lvert w_{s_{a+d}}^{-1}...w_{s_{a+1}}^{-1}\rvert-2\lvert w_{s_{k+1}}...w_{s_a}\rvert-i_{s_{a+d}}+i_{s_k}+2\lvert w_{s_{k}}...w_{s_a}\rvert-2\lvert w_{s_{a+d}}^{-1}...w_{s_{a+1}}^{-1}\rvert+2 
	= i_{s_k}+2\lvert w_{s_{k}}\rvert+2$. \\
	\indent	It remains the case that $w_r\in W_u$ for all $r\in\{{s_k}+1,...,a+b\}$. By  Lemmas \ref{lem zu=dz}, and \ref{x>u}, we obtain $d_{r_{{s_k}+1}}=i_{s_{k+1}}, m_{r_{{s_k}+1}}=i_{s_{k+1}}+2\lvert w_{s_{k+1}}...w_{s_a}\rvert$, and $m_{r_{s_k}}=i_{s_{k}}+2\lvert w_{s_{k}}...w_{s_a}\rvert$. This implies $ d_{r_{{s_k}+1}}-(m_{r_{{s_k}+1}}-m_{r_{s_k}}-2) = i_{s_{k+1}}-i_{s_{k+1}}-2\lvert w_{s_{k+1}}...w_{s_a}\rvert+i_{s_k}+2\lvert w_{s_{k}}...w_{s_a}\rvert+2 = i_{s_k}+2\lvert w_{s_{k}}\rvert+2$.
\end{proof}
\begin{lem} \label{8}
	Let $d\in\{1,...,b\}$. It holds $d_{r_{s_{a+d}}}+2=({s_{a+d}}+1)_u-b_{s_{a+d}}$. 
\end{lem}
\begin{proof} First, we consider the case that there is the least $k\in\{1,...,a\}$ such that $s_k>s_{a+d}$. Suppose  $w_{s_{a+d}+1}\in W_u$, i.e. $s_{a+d}+1=s_{k}$. Then by Lemmas \ref{lem zu=dz} and \ref{u>x}, we get
	$d_{r_{s_{a+d}}}+2=i_{s_{a+d}}+2\lvert w_{s_{a+d}}^{-1}...w_{s_{a+1}}^{-1}\rvert-2\lvert w_{s_{k}}...w_{s_a}\rvert+2$ and  $({s_{a+d}}+1)_u-b_{s_{a+d}}=i_{s_{k}}-(({s_{a+d}}+1)_x-i_{s_{a+d}}-2\lvert w_{s_{a+d}}^{-1}\rvert-2)=i_{s_{k}}-(i_{s_k}+2\lvert w_{s_{k}}...w_{s_a}\rvert-2\lvert w_{s_{a+d-1}}^{-1}...w_{s_{a+1}}^{-1}\rvert)+i_{s_{a+d}}+2\lvert w_{s_{a+d}}^{-1}\rvert+2=i_{s_{a+d}}+2\lvert w_{s_{a+d}}^{-1}...w_{s_{a+1}}^{-1}\rvert-2\lvert w_{s_{k}}...w_{s_a}\rvert+2= d_{r_{s_{a+d}}}+2$. Suppose  $w_{s_{a+d}+1}\in W_x$, i.e. $s_{a+d}+1=s_{a+d-1}$. Then by Lemmas \ref{lem zu=dz} and \ref{u>x}, we get 
	$d_{r_{s_{a+d}}}+2=i_{s_{a+d}}+2\lvert w_{s_{a+d}}^{-1}...w_{s_{a+1}}^{-1}\rvert-2\lvert w_{s_{k}}...w_{s_a}\rvert+2$ and  $({s_{a+d}}+1)_u-b_{s_{a+d}}=i_{s_{a+d-1}}+2\lvert w_{s_{a+d-1}}^{-1}...w_{s_{a+1}}^{-1}\rvert-2\lvert w_{s_{k}}...w_{s_a}\rvert-(i_{s_{a+d-1}}-i_{s_{a+d}}-2\lvert w_{s_{a+d}}^{-1}\rvert-2)
	=i_{s_{a+d}}+2\lvert w_{s_{a+d}}^{-1}...w_{s_{a+1}}^{-1}\rvert-2\lvert w_{s_{k}}...w_{s_a}\rvert+2= d_{r_{s_{a+d}}}+2$. \\
	\indent	It remains the case that $w_r\in W_x$ for all $r\in\{{s_{a+d}}+1,...,a+b\}$. By  Lemmas \ref{lem zu=dz} and \ref{u>x}, we obtain 
	$d_{r_{s_{a+d}}}+2=i_{s_{a+d}}+2\lvert w_{s_{a+d}}^{-1}...w_{s_{a+1}}^{-1}\rvert+2$ and  $({s_{a+d}}+1)_u-b_{s_{a+d}}=i_{s_{a+d-1}}+2\lvert w_{s_{a+d-1}}^{-1}...w_{s_{a+1}}^{-1}\rvert-i_{s_{a+d-1}}+i_{s_{a+d}}+2\lvert w_{s_{a+d}}^{-1}\rvert+2=i_{s_{a+d}}+2\lvert w_{s_{a+d}}^{-1}...w_{s_{a+1}}^{-1}\rvert+2= d_{r_{s_{a+d}}}+2$.
\end{proof}	
If $\rho\in dom(\alpha)$ then $\rho\notin A$ as already mentioned. So, we have $\rho\overline{v}_A=\rho$ for all $\rho\in dom(\alpha)$. Moreover, if we apply $\overline{w}_\alpha=\overline{v}_A\overline{{w}}_{s_1}...\overline{{w}}_{s_a}{\overline{w_{s_{a+1}}^{-1}}}...{\overline{w_{s_{a+b}}^{-1}}}$ to an element $\rho$ of the domain of $\alpha$, then each of the transformations $\overline{{w}}_{s_1},...,\overline{{w}}_{s_a}$ adds successively the double of the length of the corresponding word to $\rho$ or maps identical. So we obtain $\rho'=\rho\overline{v}_A\overline{{w}}_{s_1}...\overline{{w}}_{s_a}$ for some $\rho'\in \overline{n}$. On the other hand, each of the transformations ${\overline{w_{s_{a+1}}^{-1}}},...,{\overline{w_{s_{a+b}}^{-1}}}$ cancels successively the double of the length of the corresponding word from $\rho'$ or maps identical. This procedure will be clear by Remark \ref{move}, which characterizes the transformations corresponding to the words in $W_u\cup W_x$. \\
\indent By the definitions of the transformations $\overline{u}_i$ and  $\overline{x}_i$, $i\in\{1,...,n-2\}$, we can easy determine the transformation $\overline{s}$, for  all $s\in W_u\cup W_x$.
\begin{remark} \label{move}
	Let $i\in\{1,...,n-2\}$ and $j\in\{1,...,\lfloor{\frac{n-i}{2}}\rfloor\}  $.	We have: \\
	(i)	$dom(\overline{u}_i)=\{1,...,i,i+4,...,n\}$ and 
	$\rho\overline{u}_i = \left\{ \begin{array}{rcl}
	\rho+2 & \mbox{for}
	& \rho\leq i ; \\ \rho & \mbox{for} & \rho\geq i+4 .
	\end{array}\right.$ \\
	(ii) $dom(\overline{u}_{i,j})=\{1,...,i,i+2j+2,...,n\}$ and  $\rho\overline{u}_{i,j} = \left\{ \begin{array}{rcl}
	\rho+2j & \mbox{for}
	& \rho\leq i ;\\ \rho & \mbox{for} & \rho\geq i+2j+2.
	\end{array}\right.$ \\
	(iii) $dom(\overline{x}_i)=\{3,...,i+2,i+4,...,n\}$ and 
	$\rho\overline{x}_i = \left\{ \begin{array}{rcl}
	\rho-2 & \mbox{for}
	& 3\leq \rho\leq i+2 ;\\ \rho & \mbox{for} & \rho\geq i+4 .
	\end{array}\right.$  
	\\	(iv) $dom(\overline{x^{-1}_{i,j}})=\{3+2j-2,...,i+2j,i+2j+2,...,n\}$ and \\ $\rho\overline{x^{-1}_{i,j}} = \left\{ \begin{array}{rcl}
	\rho-2j & \mbox{for}
	& 3+2j-2\leq \rho\leq i+2j; \\ \rho & \mbox{for} & \rho\geq i+2j+2.
	\end{array}\right.$
\end{remark}
Lemma \ref{tool} is a technical lemma. Together with Remark \ref{move}, it will give an important tool for the calculation of $\rho \overline{w_\alpha^*}$, for any $\rho\in dom( \overline{w_\alpha^*})$.
\begin{lem} \label{tool} Let $k\in\{1,...,a\}$. Further, 
	let  $\rho \in\{i_{s_k-1}+2\lvert w_{s_k-1}\rvert+2,...,i_{s_k}\}$, whenever $w_{s_k-1}\in W_u$,  let
	$\rho \in\{i_{s_k}-b_{s_k-1},...,i_{s_k}\}$, whenever $w_{s_k-1}\in W_x$,  let 
	$\rho\in\{1,...,i_{s_k}\}$, whenever $s_k=1$ and $1_u<1_x$, and let $\rho\in\{1_u-1_x+1,...,i_{s_k}\}$, whenever $s_k=1$ and $1_u>1_x$. Then we have the following statements: \\
	(i) $\rho\geq i_{s_d}+2\lvert w_{s_d}\rvert+2$ for all $d\in\{1,...,k-1\}$. \\
	(ii) $\rho+2\lvert w_{s_k}...w_{s_{d-1}}\rvert\leq i_{s_d}$ for all $d\in\{k,...,a\}$. \\
	Let $h=max\{z\in\bar{b} :  s_k<s_{a+z}\}$ (if it exists). Then: \\
	(iii) $\rho+2\lvert w_{s_k}...w_{s_a}\rvert-2\lvert w^{-1}_{s_{a+1}}...w^{-1}_{s_{a+d-1}}\rvert\leq i_{s_{a+d}}+2\lvert w^{-1}_{s_{a+d}}\rvert$ for all $d\in\{1,...,h\}$. \\
	(iv) $\rho+2\lvert w_{s_k}...w_{s_a}\rvert-																																																																																																																															2\lvert w^{-1}_{s_{a+1}}...w^{-1}_{s_{a+d-1}}\rvert\geq i_{s_{a+d}}+2\lvert w^{-1}_{s_{a+d}}\rvert+2$ for all $d\in\{h+1,...,b\}$. \\
	(v) If $s_k>s_{a+1}$ then  $i_{s_{a+1}}+2\lvert w^{-1}_{s_{a+1}}\rvert+2\leq \rho + 2\lvert w_{s_k}...w_{s_a}\rvert$. 
\end{lem}
\begin{proof}
	(i) Let $d\in\{1,...,k-1\}$. By Lemma \ref{1-4}(ii), we have  $i_{s_{k-1}}\geq i_{s_d}+2\lvert w_{s_d}\rvert+2$. If $w_{s_k-1}\in W_u$ then $\rho\geq i_{s_k-1}+2\lvert w_{s_k-1}\rvert+2=i_{s_{k-1}}+2\lvert w_{s_{k-1}}\rvert+2$, i.e. $\rho\geq i_{s_d}+2\lvert w_{s_d}\rvert+2$.
	Suppose $w_{s_k-1}\in W_x$. By Lemma \ref{1-4}(ii, iii), we have $i_{s_d}+2\lvert w_{s_d}\rvert+2\leq i_{s_{k-1}}$ and  $i_{s_{k-1}}+2\lvert w_{s_{k-1}}\rvert+2\leq(s_k-1)_u$, respectively. This implies  $(s_k-1)_u	=(s_k)_u-b_{s_k-1}-2 <(s_k)_u-b_{s_k-1}\leq \rho$. Thus,  $i_{s_d}+2\lvert w_{s_d}\rvert+2\leq \rho$. \\
	
	\noindent	(ii) Let $d\in\{k,...,a\}$. Since $s_k<s_{k+1}$, we have $i_{s_k}+2\lvert w_{s_k}\rvert+1<i_{s_{k+1}}$ by Lemma \ref{1-4}(ii). 
	This implies  
	$i_{s_k}+2\lvert w_{s_k}\rvert+2\lvert w_{s_{k+1}}\rvert+1< i_{s_{k+1}}+2\lvert w_{s_{k+1}}\rvert
	< i_{s_{k+1}}+2\lvert w_{s_{k+1}}\rvert+1 
	< i_{s_{k+2}}$.
	After $d-k$ such steps, we have 
	$i_{s_k}+2\lvert w_{s_k}\rvert+2\lvert w_{s_{k+1}}\rvert+...+2\lvert w_{s_{d-1}}\rvert+1<i_{s_d}$. Since $\rho\leq i_{s_k}$, we obtain $\rho+2\lvert w_{s_k}...w_{s_{d-1}}\rvert\leq i_{s_d}.$ \\ 
	
	\noindent	(iii) Let $d\in\{1,...,h\}$ and let $c$ be the greatest $f\in\{0,...,a-k\}$ such that $s_{k+f}<s_{a+h}$. We put $m=s_{k+c}$. By Lemma \ref{lem zu=dz}, we have $m_u=i_{s_{k+c}}$ and $m_x=(m+1)_x-(m+1)_u+i_{s_{k+c}}+2\lvert w_{s_{k+c}}\rvert$. Note that $w_{s_{a+h}}=w_{m+1}\in W_x$. So we have $(m+1)_u=(m+2)_u-(m+2)_x+i_{s_{a+h}}+2\lvert w_{s_{a+h}}^{-1}\rvert$ and $(m+1)_x=i_{s_{a+h}}$. By Lemma \ref{1-4}(iii), we get
	$i_{s_{k+c}}+2\lvert w_{s_{k+c}}\rvert<(m+1)_u=(m+2)_u-(m+2)_x+i_{s_{a+h}}+2\lvert w_{s_{a+h}}^{-1}\rvert$, i.e. 
	\begin{equation}
	i_{s_{k+c}}+2\lvert w_{s_{k+c}}\rvert-(m+2)_u+(m+2)_x<i_{s_{a+h}}+2\lvert w_{s_{a+h}}^{-1}\rvert. 
	\end{equation}   
	If $w_{m+2}\in W_u$ then $w_{m+2}=w_{s_{k+c+1}}$ and (1) provide   
	$i_{s_{k+c}}+2\lvert w_{s_{k+c}}\rvert-i_{s_{k+c+1}}+(m+3)_x-(m+3)_u+i_{s_{k+c+1}}+2\lvert w_{s_{k+c+1}}\rvert<i_{s_{a+h}}+2\lvert w_{s_{a+h}}^{-1}\rvert$, i.e. $i_{s_{k+c}}+2\lvert w_{s_{k+c}}w_{s_{k+c+1}}\rvert+(m+3)_x-(m+3)_u<i_{s_{a+h}}+2\lvert w_{s_{a+h}}^{-1}\rvert$. 
	If $w_{m+2}\in W_x$ then $(1)$ provides $i_{s_{k+c}}+2\lvert w_{s_{k+c}}\rvert+(m+3)_x-(m+3)_u-2\lvert w_{s_{a+h-1}}^{-1}\rvert <i_{s_{a+h}}+2\lvert w_{s_{a+h}}^{-1}\rvert$.
	We repeat this procedure until $w_{a+b}$. 
	If $w_{a+b}\in W_u$ then $a+b=s_a$, i.e. $(a+b)_u=i_{s_a}$ and $(a+b)_x=i_{s_a}+2\lvert w_{s_a}\rvert$.
	If $w_{a+b}\in W_x$ then $a+b=s_{a+1}$, i.e. $(a+b)_u=i_{s_{a+1}}+2\lvert w_{s_{a+1}}^{-1}\rvert$ and $(a+b)_x=i_{s_{a+1}}$.
	This provides $i_{s_{k+c}}+2\lvert w_{s_{k+c}}...w_{s_a}\rvert-2\lvert w_{s_{a+h-1}}^{-1}...w_{s_{a+1}}^{-1}\rvert\leq i_{s_{a+h}}+2\lvert w_{s_{a+h}}^{-1}\rvert$. If $c=0$, then we have the required inequality.
	If $c>0$ then 
	$i_{s_k}+2\lvert w_{s_k}\rvert<i_{s_{k+1}},
	i_{s_k}+2\lvert w_{s_k}\rvert+2\lvert w_{s_{k+1}}\rvert< i_{s_{k+2}}, ... ,
	i_{s_k}+2\lvert w_{s_k}...w_{s_{k+c-1}}\rvert< i_{s_{k+c}}$ and we can conclude $i_{s_{k}}+2\lvert w_{s_{k}}...w_{s_a}\rvert-2\lvert w_{s_{a+h-1}}^{-1}...w_{s_{a+1}}^{-1}\rvert\leq i_{s_{a+h}}+2\lvert w_{s_{a+h}}^{-1}\rvert$. Altogether, we have shown that $\rho+2\lvert w_{s_{k}}...w_{s_a}\rvert-2\lvert w_{s_{a+h-1}}^{-1}...w_{s_{a+1}}^{-1}\rvert\leq i_{s_{a+h}}+2\lvert w_{s_{a+h}}^{-1}\rvert$ since $\rho\leq i_{s_k}$. If $d=h$ then the statement holds. On the other hand, if $d<h$ then we have $i_{s_{a+h}}+2\lvert w_{s_{a+h}}^{-1}\rvert<i_{s_{a+h}}+2\lvert w_{s_{a+h}}^{-1}\rvert+1<i_{s_{a+d}}<i_{s_{a+d}}+2\lvert w_{s_{a+d}}^{-1}\rvert$. This completes the proof. \\
	
	\noindent (iv) Let $d\in\{h+1,...,b\}$. Note that $s_{a+h+1}<a+b$. We put $m=s_{a+h+1}$. Then we obtain $m+1=s_t$, where $t=min\{z\in\{1,...,a\} : s_{a+h+1}<s_z\}$. 
	We have $m_u=(m+1)_u-(m+1)_x+i_{s_{a+h+1}}+2\lvert w_{s_{a+h+1}}^{-1}\rvert$ and $m_x=i_{s_{a+h+1}}$. Because $w_{m+1}=w_{s_t}\in W_u$, we have $(m+1)_u=i_{s_t}$ and $(m+1)_x=(m+2)_x-(m+2)_u+i_{s_t}+2\lvert w_{s_t}\rvert$. 
	We get $m_x+2\lvert w_m\rvert+2\leq (m+1)_x$, i.e.
	\begin{equation}
	i_{s_{a+h+1}}+2\lvert w_{s_{a+h+1}}^{-1}\rvert+2\leq (m+2)_x-(m+2)_u+i_{s_t}+2\lvert w_{s_t}\rvert.
	\end{equation} 
	\indent If $w_{m+2}\in W_u$ then $m+2=s_{t+1}$ and (2) provide $i_{s_{a+h+1}}+2\lvert w_{s_{a+h+1}}^{-1}\rvert+2\leq (m+3)_x-(m+3)_u+i_{s_{t+1}}+2\lvert w_{s_{t+1}}\rvert-i_{s_{t+1}}+2\lvert w_{s_t}\rvert+i_{s_t}=i_{s_t}+2\lvert w_{s_t}\rvert+2\lvert w_{s_{t+1}}\rvert+(m+3)_x-(m+3)_u$. \\
	\indent If $w_{m+2}\in W_x$ then $m+2=s_{a+h}$ and (2) provide $i_{s_{a+h+1}}+2\lvert w_{s_{a+h+1}}^{-1}\rvert+2\leq i_{s_{a+h}}- (m+3)_u+(m+3)_x-i_{s_{a+h}}-2\lvert w_{s_{a+h}}^{-1}\rvert+2\lvert w_{s_t}\rvert+i_{s_t}=i_{s_t}+2\lvert w_{s_t}\rvert-2\lvert w_{s_{a+h}}^{-1}\rvert+(m+3)_x-(m+3)_u$.
	We repeat this procedure until $w_{a+b}$. \\
	\indent If $w_{a+b}\in W_u$ then $a+b=s_a$, i.e. $(a+b)_u=i_{s_a}$ and $(a+b)_x=i_{s_a}+2\lvert w_{s_a}\rvert$. \\
	\indent If $w_{a+b}\in W_x$ then $a+b=s_{a+1}$, i.e. $(a+b)_u=i_{s_{a+1}}+2\lvert w_{s_{a+b}}^{-1}\rvert$ and $(a+b)_x=i_{s_{a+1}}$.
	This provides 
	\begin{equation}
	i_{s_{a+h+1}}+2\lvert w_{s_{a+h+1}}^{-1}\rvert+2\leq i_{s_t}+2\lvert w_{s_t}...w_{s_a}\rvert-2\lvert w_{s_{a+h}}^{-1}...w_{s_{a+1}}^{-1}\rvert.
	\end{equation} 
	We have $i_{s_{a+d}}+2\lvert w_{s_{a+d}}^{-1}\rvert+2\leq i_{s_{a+d-1}}$. This gives $ 
	i_{s_{a+d}}+2\lvert w_{s_{a+d}}^{-1}w_{s_{a+d-1}}^{-1}\rvert+2 \leq  i_{s_{a+d-1}}+2\lvert w_{s_{a+d-1}}^{-1}\rvert+2 \leq i_{s_{a+d-2}}$. After $d-h$ such steps, we will obtain 
	\begin{align}
	i_{s_{a+d}}+2\lvert w_{s_{a+d}}^{-1}...w_{s_{a+h+1}}^{-1}\rvert+2 \leq i_{s_{a+h+1}}+2\lvert w_{s_{a+h+1}}^{-1}\rvert+2.
	\end{align}
	Then (3) and (4) provide (by transitivity) \\
	$i_{s_{a+d}}+2\lvert w_{s_{a+d}}^{-1}\rvert+2\lvert w_{s_{a+d-1}}^{-1}...w_{s_{a+h+1}}^{-1}\rvert \leq i_{s_t}+2\lvert w_{s_t}...w_{s_{a}}\rvert-2\lvert w_{s_{a+h}}^{-1}...w_{s_{a+1}}^{-1}\rvert$   and
	\begin{align}
	i_{s_{a+d}}+2\lvert w_{s_{a+d}}^{-1}\rvert+2 &\leq i_{s_t}+2\lvert w_{s_t}...w_{s_{a}}\rvert-2\lvert w_{s_{a+d-1}}^{-1}...w_{s_{a+1}}^{-1}\rvert 
	\end{align}
	\indent   Suppose $t=k$. If $\rho=(m+1)_u=i_{s_t}$ then the proof is finished by (3). Suppose $\rho<(m+1)_u$. We show that $i_{s_{a+h+1}} +2\lvert w_{s_{a+h+1}}^{-1}\rvert+2\leq (m+1)_u-b_m+2\lvert w_{s_k}...w_{s_a}\rvert-2\lvert w_{s_{a+h}}^{-1}...w_{s_{a+1}}^{-1}\rvert$. 
	We have $b_m = (m+1)_x-(i_{s_{a+h+1}}+2\lvert w_{s_{a+h+1}}^{-1}\rvert+2)$.
	This provides $i_{s_k}-((m+1)_u-b_m) = i_{s_k}+2\lvert w_{s_k}...w_{s_a}\rvert-2\lvert w_{s_{a+h}}^{-1}...w_{s_{a+1}}^{-1}\rvert-(i_{s_{a+h+1}}+2\lvert w_{s_{a+h+1}}^{-1}\rvert+2)$ by Lemmas \ref{lem zu=dz} and \ref{x>u} and since $s_{a+h+1}<s_k<s_{a+h}$. Then   
	$i_{s_{a+h+1}}+2\lvert w_{s_{a+h+1}}^{-1}\rvert+2= (m+1)_u-b_m+2\lvert w_{s_k}...w_{s_a}\rvert -2\lvert w_{s_{a+h}}^{-1}...w_{s_{a+1}}^{-1}\rvert \leq \rho+2\lvert w_{s_k}...w_{s_a}\rvert-2\lvert w_{s_{a+h}}^{-1}...w_{s_{a+1}}^{-1}\rvert$.
	By (4) and transitivity, we get
	$i_{s_{a+d}}+2\lvert w_{s_{a+d}}^{-1}\rvert+2\lvert w_{s_{a+d-1}}^{-1}...w_{s_{a+h+1}}^{-1}\rvert+2 \leq \rho+2\lvert w_{s_k}...w_{s_a}\rvert-2\lvert w_{s_{a+h}}^{-1}...w_{s_{a+1}}^{-1}\rvert$. Thus, $i_{s_{a+d}}+2\lvert w_{s_{a+d}}^{-1}\rvert+2 \leq \rho+2\lvert w_{s_k}...w_{s_a}\rvert-2\lvert w_{s_{a+d-1}}^{-1}...w_{s_{a+1}}^{-1}\rvert$. \\
	\indent  If $k=t+1$ then $\rho\in\{i_{s_t}+2\lvert w_{s_t}\rvert +2,...,i_{s_k}\}$ and by (5), we can conclude that $
	i_{s_{a+d}}+2\lvert w_{s_{a+d}}^{-1}\rvert+2 \leq i_{s_t}+2\lvert w_{s_t}...w_{s_{a}}\rvert-2\lvert w_{s_{a+d-1}}^{-1}...w_{s_{a+1}}^{-1}\rvert< i_{s_t}+2\lvert w_{s_t}\rvert+2+2\lvert w_{s_k}...w_{s_{a}}\rvert-2\lvert w_{s_{a+d-1}}^{-1}...w_{s_{a+1}}^{-1}\rvert \leq \rho+2\lvert w_{s_k}...w_{s_{a}}\rvert-2\lvert w_{s_{a+d-1}}^{-1}...w_{s_{a+1}}^{-1}\rvert$. \\
	\indent  If $k\geq t+2$ then  $\rho\in\{i_{s_{k-1}}+2\lvert w_{s_{k-1}}\rvert+2,...,i_{s_k}\}$ and by Lemma \ref{1-4}(ii), we obtain $
	i_{s_t}+2\lvert w_{s_t}\rvert+2 \leq i_{s_{t+1}}, 
	i_{s_{t+1}}+2\lvert w_{s_{t+1}}\rvert+2\leq i_{s_{t+2}}, ... ,
	i_{s_{k-2}}+2\lvert w_{s_{k-2}}\rvert+2\leq i_{s_{k-1}}$. This implies $i_{s_t}+2\lvert w_{s_t}...w_{s_{k-1}}\rvert\leq i_{s_{k-1}}+2\lvert w_{s_{k-1}}\rvert+2$, i.e. 
	$i_{s_t}+2\lvert w_{s_t}...w_{s_{k-1}}\rvert+2\lvert w_{s_k}...w_{s_a}\rvert\leq i_{s_{k-1}}+2\lvert w_{s_{k-1}}\rvert+2+2\lvert w_{s_k}...w_{s_a}\rvert$ and      $   
	i_{s_t}+2\lvert w_{s_t}...w_{s_a}\rvert-2\lvert w_{s_{a+d-1}}^{-1}...w_{s_{a+1}}^{-1}\rvert \leq i_{s_{k-1}}+2\lvert w_{s_{k-1}}\rvert+2+2\lvert w_{s_k}...w_{s_a}\rvert-2\lvert w_{s_{a+d-1}}^{-1}...w_{s_{a+1}}^{-1}\rvert
	\leq\rho+2\lvert w_{s_k}...w_{s_a}\rvert-2\lvert w_{s_{a+d-1}}^{-1}...w_{s_{a+1}}^{-1}\rvert$.  
	By (5) and transitivity, we get $i_{s_{a+d}}+2\lvert w_{s_{a+d}}^{-1}\rvert+2\leq \rho+2\lvert w_{s_k}...w_{s_{a}}\rvert-2\lvert w_{s_{a+d-1}}^{-1}...w_{s_{a+1}}^{-1}\rvert$. \\
	
	\noindent  (v) It can be proved similar to (iv).
\end{proof}

Now we are able to calculate $\rho\overline{w}_\alpha$ for all $\rho$ in the domain of $\overline{w}_\alpha$.
\begin{prop} \label{9}
	Let $k\in\{2,...,a\}$ with $w_{s_k-1}\in W_u$ and let $\hat{\rho}\in\{0,...,a_{s_k-1}\}$ such that $d_{r_{s_k}}-\hat{\rho}\in dom (\alpha)$. We have $(d_{r_{s_k}}-\hat{\rho})\overline{w}_\alpha=(d_{r_{s_k}}-\hat{\rho})\alpha$.
\end{prop}
\begin{proof}
	We have $(s_k)_u=i_{s_k}=d_{r_{s_k}}$ by Lemma \ref{lem zu=dz}. It is easy to verify that  $d_{r_{s_k}}-\hat{\rho}\notin A$ by the definition of set $A$ and we  observe that $d_{r_{s_k}}-\hat{\rho}\in \{i_{s_k-1}+2j_{s_k-1}+2,...,i_{s_k}\}$. Moreover, we have 
	$a_{s_k-1}\\
	=(s_k)_u-i_{s_k-1}-2j_{s_k-1}-2 \\   
	= (s_k)_u-(d_{r_{q}}-m_{r_{q}}+m_{r_{q-1}}+2)$ \indent  (by Lemma \ref{A tran= word})     \\
	$	= d_{r_q}-(d_{r_{q}}-m_{r_{q}}+m_{r_{q-1}}+2)$  \quad \quad \ \ (by Lemma \ref{lem zu=dz}) \\
	$=m_{r_q}-m_{r_{q-1}}-2.$
	Thus, $\hat{\rho}\in\{0,...,a_{s_k-1}\}=\{0,...,m_{r_q}-m_{r_{q-1}}-2\}$. 
	This implies $(d_{r_{s_k}}-\hat{\rho})\alpha=m_{r_{s_k}}-\hat{\rho}$. We will show $(d_{r_{s_k}}-\hat{\rho})\overline{w}_\alpha=m_{r_{s_k}}-\hat{\rho}$. We have \\
	$ (d_{r_{s_k}}-\hat{\rho})\overline{v}_A\overline{w}_{s_1}...\overline{w}_{s_a}\overline{w_{s_{a+1}}^{-1}}...\overline{w_{s_{a+b}}^{-1}}$ \\ = $(d_{r_{s_k}}-\hat{\rho})\overline{w}_{s_1}...\overline{w}_{s_a}\overline{w_{s_{a+1}}^{-1}}...\overline{w_{s_{a+b}}^{-1}}$ \quad \ \quad (since $d_{r_{s_k}}-\hat{\rho}\notin A$) \\
	$=(d_{r_{s_k}}-\hat{\rho})\overline{w}_{s_k}...\overline{w}_{s_a}\overline{w_{s_{a+1}}^{-1}}...\overline{w_{s_{a+b}}^{-1}}$ \quad\  \quad (by Remark \ref{move} and Lemma \ref{tool}(i)) \\
	$= (i_{s_k}-\hat{\rho}+2\lvert{w}_{s_k}...{w}_{s_a}\rvert)\overline{w_{s_{a+1}}^{-1}}...\overline{w_{s_{a+b}}^{-1}} \quad (\mbox{by Remark \ref{move} and Lemma \ref{tool}(ii)}$). \\
	If $s_{a+1}<s_k$ then $\rho+2\lvert{w}_{s_k}...{w}_{s_a}\rvert\geq i_{s_{a+1}}+2\lvert w_{s_{a+1}}^{-1}\rvert+2$  by Lemma \ref{tool}(v). 
	So, we obtain
	$(i_{s_k}-\hat{\rho}+2\lvert{w}_{s_k}...{w}_{s_a}\rvert)\overline{w_{s_{a+1}}^{-1}}...\overline{w_{s_{a+b}}^{-1}} \\ =i_{s_k}-\hat{\rho}+2\lvert{w}_{s_k}...{w}_{s_a}\rvert$ \quad (by Remark \ref{move} and Lemma \ref{tool}(v)) \\ 
	$=m_{r_{s_k}}-\hat{\rho}$ \indent \qquad \ \indent \ (by Lemma \ref{x>u}). \\
	If $\overline{w}_\alpha = \overline{w}_{s_1}...\overline{w}_{s_a}$ then we obtain $(d_{r_{s_k}}-\hat{\rho})\overline{w}_\alpha= m_{r_{s_k}}-\hat{\rho}$ by similar arguments. \\
	\indent Suppose now there is the greatest $h\in\{1,...,b\}$ such that $s_{a+h}>s_k$. 
	For $h=b$, Remark \ref{move} and Lemma \ref{x>u} and \ref{tool}(iii) provide $(i_{s_k}-\hat{\rho}+2\lvert {w}_{s_k}...{w}_{s_a}\rvert)\overline{w_{s_{a+1}}^{-1}}...\overline{w_{s_{a+b}}^{-1}}=i_{s_k}-\hat{\rho}+2\lvert w_{s_k}...w_{s_a}\rvert-2\lvert {w}_{s_{a+1}}^{-1}...{w}_{s_{a+b}}^{-1}\rvert=m_{r_{s_k}}-\hat{\rho}$. It remains the case  $h\in\{1,...,b-1\}$. Here, we have $(i_{s_k}-\hat{\rho}+2\lvert {w}_{s_k}...{w}_{s_a}\rvert)\overline{w_{s_{a+1}}^{-1}}...\overline{w_{s_{a+b}}^{-1}}=(i_{s_k}-\hat{\rho}+2\lvert w_{s_k}...w_{s_a}\rvert-2\lvert {w}_{s_{a+1}}^{-1}...{w}_{s_{a+h}}^{-1}\rvert)\overline{w_{s_{a+h+1}}^{-1}}...\overline{w_{s_{a+b}}^{-1}}$ by Remark \ref{move} and Lemma \ref{tool}(iii). Since $s_{a+h+1}<s_k$ by Remark \ref{move} and Lemma \ref{x>u} and \ref{tool}(iv), we obtain $ (i_{s_k}-\hat{\rho}+2\lvert w_{s_k}...w_{s_a}\rvert-2\lvert {w}_{s_{a+1}}^{-1}...{w}_{s_{a+h}}^{-1}\rvert)\overline{w_{s_{a+h+1}}^{-1}}...\overline{w_{s_{a+b}}^{-1}}=i_{s_k}-\hat{\rho}+2\lvert w_{s_k}...w_{s_a}\rvert-2\lvert {w}_{s_{a+1}}^{-1}...{w}_{s_{a+h}}^{-1}\rvert= m_{r_{s_k}}-\hat{\rho}$.    
\end{proof}

\begin{prop} \label{10}
	Let $k\in\{1,...,a\}$ with $w_{s_k-1}\in W_x$ and let $\hat{\rho}\in\{0,...,b_{s_k-1}\}$ such that $d_{r_{s_k}}-\hat{\rho}\in dom(\alpha)$. We have $(d_{r_{s_k}}-\hat{\rho})\overline{w}_\alpha=(d_{r_{s_k}}-\hat{\rho})\alpha$.
\end{prop}
\begin{proof}
	There is $h\in\{1,...,b\}$ with $w_{s_k-1}=w_{s_{a+h}}$. By Lemma \ref{lem zu=dz}, we have $i_{s_k}=(s_k)_u=d_{r_{s_k}}$. Since $d_{r_{s_k}}-\hat{\rho}\in dom (\alpha)$, we can conclude $d_{r_{s_k}}-\hat{\rho}\notin A$ by the definition of set $A$. We observe that $d_{r_{s_k}}-\hat{\rho}\in \{(s_k)_u-b_{s_k-1},...,i_{s_k}\}$. Moreover, we have 
	$b_{s_k-1}=(s_k)_u-((s_k)_u-b_{s_k-1})
	= d_{r_{s_k}}-(d_{r_{s_k-1}}+2)$  by Lemmas \ref{lem zu=dz} and \ref{8}.
	Thus, $\hat{\rho}\in\{0,...,b_{s_k-1}\}=\{0,...,d_{r_{s_k}}-(d_{r_{s_k-1}}+2)\}$. 
	This implies $(d_{r_{s_k}}-\hat{\rho})\alpha=m_{r_{s_k}}-\hat{\rho}$.  We will show $(d_{r_{s_k}}-\hat{\rho})\overline{w}_\alpha=m_{r_{s_k}}-\hat{\rho}$. We have 
	$ (d_{r_{s_k}}-\hat{\rho})\overline{v}_A\overline{w}_{s_1}...\overline{w}_{s_a}\overline{w_{s_{a+1}}^{-1}}...\overline{w_{s_{a+b}}^{-1}}$  \\ = $(d_{r_{s_k}}-\hat{\rho})\overline{w}_{s_1}...\overline{w}_{s_a}\overline{w_{s_{a+1}}^{-1}}...\overline{w_{s_{a+b}}^{-1}}$ \quad \quad \quad \indent (since $d_{r_{s_k}}-\hat{\rho}\notin A$) \\
	$=(d_{r_{s_k}}-\hat{\rho})\overline{w}_{s_k}...\overline{w}_{s_a}\overline{w_{s_{a+1}}^{-1}}...\overline{w_{s_{a+b}}^{-1}}$ \quad \quad \quad \indent (by Remark \ref{move} and Lemma \ref{tool}(i)) \\
	$= (i_{s_k}-\hat{\rho}+2\lvert {w}_{s_k}...{w}_{s_a}\rvert)\overline{w_{s_{a+1}}^{-1}}...\overline{w_{s_{a+b}}^{-1}}$ \quad \quad \quad (by Remark \ref{move} and Lemma \ref{tool}(ii)). \\
	Suppose $h>1$. Then 
	$(i_{s_k}-\hat{\rho}+2\lvert{w}_{s_k}...{w}_{s_a}\rvert)\overline{w_{s_{a+1}}^{-1}}...\overline{w_{s_{a+b}}^{-1}}$ \\
	$=(i_{s_k}-\hat{\rho}+2\lvert w_{s_k}...w_{s_a}\rvert-2\lvert{w}_{s_{a+1}}^{-1}...{w}_{s_{a+h-1}}^{-1}\rvert)\overline{w_{s_{a+h}}^{-1}}...\overline{w_{s_{a+b}}^{-1}}$   (by Remark \ref{move} and Lemma \ref{tool}(iii)) \\ $=i_{s_k}-\hat{\rho}+2\lvert w_{s_k}...w_{s_a}\rvert-2\lvert{w}_{s_{a+1}}^{-1}...{w}_{s_{a+h-1}}^{-1}\rvert$     \ \indent \indent \indent \ \ (by Remark \ref{move} and Lemma \ref{tool}(iv)) \\
	$	= m_{r_{s_k}}-\hat{\rho}$   \ \quad \quad \indent \indent \indent \indent \indent \indent \indent \indent \indent \indent   (by Lemma \ref{x>u}). \\ 
	If $h=1$ then  $(i_{s_k}-\hat{\rho}+2\lvert{w}_{s_k}...{w}_{s_a}\rvert)\overline{w_{s_{a+1}}^{-1}}...\overline{w_{s_{a+b}}^{-1}}=i_{s_k}-\hat{\rho}+2\lvert{w}_{s_k}...{w}_{s_a}\rvert= m_{r_{s_k}}-\hat{\rho}$ by Remark \ref{move}, Lemmas \ref{x>u} and \ref{tool}(v).
\end{proof} 
It is clear that $\alpha^{-1}=\bigl(\begin{smallmatrix}
m_1 & < & \cdots & < & m_p \\
d_1 &  &  \cdots &    & d_p
\end{smallmatrix}\bigr)$. We consider the word $w_{\alpha^{-1}}^*$ which is defined in the same way as the word $w_\alpha^*$. We  can conclude that $\overline{w_{\alpha^{-1}}^*}$  is the inverse transformation of $\overline{w_\alpha^*}$.
\begin{lem}\label{w alpha} We have $\overline{w_{\alpha^{-1}}^*}=(\overline{w_{\alpha}^*})^{-1}.$
\end{lem}
\begin{proof}
	Let $\alpha^{-1}=\bigl(\begin{smallmatrix}
	m_1 & < & \cdots & < & m_p \\
	d_1 &  &  \cdots &    & d_p
	\end{smallmatrix}\bigr)=\bigl(\begin{smallmatrix}
	d'_1 & < & \cdots & < & d'_p \\
	m'_1 &  &  \cdots &    & m'_p
	\end{smallmatrix}\bigr)$. Then we define $r'_1,...,r'_{l'+1}\in\{1,...,p\}$ and  $w'_1,...,w'_{l'},w'_{l'+1}$ like in the construction of the word $w_\alpha^*$. It is easy to verify that $l=l'$ and $r_q=r'_{q}$ for all $q\in\{1,...,l+1\}$. Let $q\in\{1,...,l\}$ and suppose $w_q=x_{i_q,j_q}$. We have  $w_q=x_{m_{r_q},\frac{(m_{r_q+1}-m_{r_q})-(d_{r_q+1}-d_{r_q})}{2}}$ and $m_{r_q+1}-m_{r_q}>d_{r_q+1}-d_{r_q} $, i.e. $d'_{r_q+1}-d'_{r_q}>m'_{r_q+1}-m'_{r_q}$. This implies  $w'_q=u_{d'_{r_q},\frac{(d'_{r_q+1}-d'_{r_q})-(m'_{r_q+1}-m'_{r_q})}{2}} =u_{m_{r_q},\frac{(m_{r_q+1}-m_{r_q})-(d_{r_q+1}-d_{r_q})}{2}}=u_{i_q,j_q}$. For $w_q=u_{i_q,j_q}$, we obtain $w'_q=x_{i_q,j_q}$ dually. If $w_{l+1}$ is not the empty word, i.e. $w_{l+1}=x_{i_{l+1},j_{l+1}}$ or $w_{l+1}=u_{i_{l+1},j_{l+1}}$, then we obtain $w'_{l+1}=u_{i_{l+1},j_{l+1}}$ and $w'_{l+1}=x_{i_{l+1},j_{l+1}}$, respectively by  similar arguments. Thus, we get $w_{\alpha^{-1}}^*=u_{i_{s_{a+b}},j_{s_{a+b}}}... u_{i_{s_{a+1}},j_{s_{a+1}}}x^{-1}_{i_{s_{a}},j_{s_{a}}}...x^{-1}_{i_{s_{1}},j_{s_{1}}}$. We have	$\overline{w_{\alpha^{-1}}^*}= \overline{u}_{i_{s_{a+b}},j_{s_{a+b}}}...\overline{u}_{i_{s_{a+1}},j_{s_{a+1}}}\overline{x_{i_{s_{a}},j_{s_{a}}}^{-1}}...\overline{x_{i_{s_{1}},j_{s_{1}}}^{-1}}  =  (\overline{x_{i_{s_{a+b}},j_{s_{a+b}}}^{-1}})^{-1}...(\overline{x_{i_{s_{a+1}},j_{s_{a+1}}}^{-1}})^{-1} (\overline{u}_{i_{s_{a}},j_{s_{a}}})^{-1}...(\overline{u}_{i_{s_{1}},j_{s_{1}}})^{-1}\\ =  (\overline{u}_{i_{s_{1}},j_{s_{1}}}...\overline{u}_{i_{s_{a}},j_{s_{a}}}\overline{x_{i_{s_{a+1}},j_{s_{a+1}}}^{-1}}...\overline{x_{i_{s_{a+b}},j_{s_{a+b}}}^{-1}})^{-1} =(\overline{{w}}_{s_1}...\overline{{w}}_{s_a}{\overline{w_{s_{a+1}}^{-1}}}...{\overline{w_{s_{a+b}}^{-1}}} )^{-1}= (\overline{w_{\alpha}^*})^{-1}.$ 	
\end{proof}
We have proved the previous lemmas and propositions for a fixed but arbitrary transformation $\alpha$. So, we can replace $\alpha$ by $\alpha^{-1}$ and obtain the same results if we adapt the concepts defined for $\alpha$ to $\alpha^{-1}$. But this is 
straight forward in the case $\alpha^{-1}$. Using this idea, the following proposition is easy to prove.
\begin{prop} \label{12} Let  $h\in\{1,...,b\}$ with $s_{a+h}\geq 2$ and let 
	\begin{align*}
	\hat{\rho} \in \left\{ \begin{array}{rcl}
	\{0,...,a_{s_{a+h}-1}\} & \mbox{if} \ w_{s_{a+h}-1}\in W_u; \\
	\{0,...,b_{s_{a+h}-1}\} & \mbox{if} \ w_{s_{a+h}-1}\in W_x .
	\end{array}\right. 
	\end{align*}
	Then $(d_{r_{s_{a+h}}}-\hat{\rho})\overline{w}_\alpha=m_{r_{s_{a+h}}}-\hat{\rho}=(d_{r_{s_{a+h}}}-\hat{\rho})\alpha$.
\end{prop}
\begin{proof}  We have $\alpha^{-1}=\bigl(\begin{smallmatrix}
	m_1 & < & \cdots & < & m_p \\
	d_1 &  &  \cdots &    & d_p
	\end{smallmatrix}\bigr)$ and obtain the words $v_{A'}w_{\alpha^{-1}}^*$ and $w'=w_1'...w'_lw'_{l+1}$  from $\alpha^{-1}$ by the same construction as  the words $v_{A}w_{\alpha}^*$  and $w=w_1...w_{l+1}$ from $\alpha$, respectively. It is easy to verify that $w'_{s_{a+h}}\in W_u$.  We show that $(d_{r_{s_{a+h}}}-\hat{\rho})\overline{w}_\alpha=m_{r_{s_{a+h}}}-\hat{\rho}$. By Propositions \ref{9} and \ref{10}, respectively, and Lemma \ref{w alpha}, we have $(m_{r_{s_{a+h}}}-\hat{\rho})\overline{v}_{A'} \overline{w_{\alpha^{-1}}^*}=d_{r_{s_{a+h}}}-\hat{\rho}=(m_{r_{s_{a+h}}}-\hat{\rho})\overline{v}_{A'} (\overline{w_{\alpha}^*})^{-1}$. Since $m_{r_{s_{a+h}}}-\hat{\rho}\in dom(\alpha^{-1})$ and $d_{r_{s_{a+h}}}-\hat{\rho}\in dom(\alpha)$, we can conclude that $m_{r_{s_{a+h}}}-\hat{\rho}\notin A'$ and $d_{r_{s_{a+h}}}-\hat{\rho}\notin A $, respectively. Then 
	$(m_{r_{s_{a+h}}}-\hat{\rho})\overline{v}_{A'} (\overline{w_{\alpha}^*})^{-1}=d_{r_{s_{a+h}}}-\hat{\rho}$
	implies $(m_{r_{s_{a+h}}}-\hat{\rho})(\overline{w_{\alpha}^*})^{-1}=(d_{r_{s_{a+h}}}-\hat{\rho})\overline{v}_A$. We apply $\overline{w_{\alpha}^*}$ to that equation and obtain 
	$ m_{r_{s_{a+h}}}-\hat{\rho}  
	= (d_{r_{s_{a+h}}}-\hat{\rho})\overline{v}_A\overline{w_{\alpha}^*}$, i.e. $ m_{r_{s_{a+h}}}-\hat{\rho}  
	= (d_{r_{s_{a+h}}}-\hat{\rho})\overline{w}_\alpha$.
\end{proof}
It remains to consider all  $\rho$ in domain $\alpha$ with $\rho\leq d_{r_1}$ and $\rho > d_{r_l}$, respectively.

\begin{prop} \label{mp=dp} If $w_{l}=w_{s_a}$ and $m_p=d_p$ then  $\rho\overline{w}_\alpha=\rho\alpha=\rho$ for all $\rho\in\{i_{s_a}+2\lvert w_{s_a}\rvert+2,...,n\}\cap dom(\alpha)$. 
\end{prop}
\begin{proof}
	Let $\rho\in\{i_{s_a}+2\lvert w_{s_a}\rvert+2,...,n\}\cap dom(\alpha)$.	First, we  show that $m_{r_{l}}+2=i_{s_a}+2\lvert w_{s_a}\rvert+2$. In fact by Lemma \ref{lem zu=dz}, we have $m_{r_{l}}+2=l_x+2=i_{l}+2\lvert w_{l}\rvert+2=i_{s_a}+2\lvert w_{s_a}\rvert+2$. 
	Therefore, $\rho\in\{m_{r_{l}}+2,...,n\}$ and we have $\rho\alpha=\rho$. Since $\rho\in dom(\alpha)$, we can conclude $\rho\notin A$ by the definition of set $A$. Hence, we get 
	$	\rho\overline{w}_\alpha =\rho\overline{v}_A\overline{w}_{s_1}...\overline{w}_{s_a}\overline{w_{s_{a+1}}^{-1}}...\overline{w_{s_{a+b}}^{-1}}=\rho\overline{w}_{s_1}...\overline{w}_{s_a}\overline{w_{s_{a+1}}^{-1}}...\overline{w_{s_{a+b}}^{-1}}$.  Then we have 
	$\rho\overline{w}_{s_1}...\overline{w}_{s_a}\overline{w_{s_{a+1}}^{-1}}...\overline{w_{s_{a+b}}^{-1}}= \rho\overline{w_{s_{a+1}}^{-1}}...\overline{w_{s_{a+b}}^{-1}}$ by Remark  \ref{move} and Lemma \ref{tool}(i).
	By Lemma \ref{tool}(v), we can conclude $
	i_{s_{a+1}}+2\lvert w_{s_{a+1}}^{-1}\rvert+2 \leq  i_{s_a}+2\lvert w_{s_a}\rvert 
	< i_{s_a}+2\lvert w_{s_a}\rvert+2 
	\leq \rho$, which provides $\rho\overline{w_{s_{a+1}}^{-1}}...\overline{w_{s_{a+b}}^{-1}}=\rho$ by Remark \ref{move}. Thus, $\rho\overline{w}_\alpha=\rho=\rho\alpha$.
\end{proof}
Similarly, we can prove:
\begin{prop} \label{mp'}  If $w_{l}=w_{s_{a+1}}$ and 	
	$m_p=d_p$ then  $\rho\overline{w}_\alpha=\rho\alpha=\rho$ for all $\rho\in\{i_{s_{a+1}}+2\lvert w_{s_{a+1}}\rvert+2,...,n\}\cap dom(\alpha)$.
\end{prop}
\begin{prop}\label{begin} Let  $\hat{\rho}\in\{0,...,min\{1_u,1_x\}-1\}$ such that $d_{r_1}-\hat{\rho}\in dom(\alpha)$. Then we have $(d_{r_1}-\hat{\rho})\overline{w}_\alpha=(d_{r_1}-\hat{\rho})\alpha$.
\end{prop}
\begin{proof} Recall that $1_u=d_{r_1}$ and $1_x=m_{r_1}$ by Lemma \ref{lem zu=dz}. \\
	\indent	If $1_u\leq 1_x$ then $\hat{\rho}\in\{0,...,1_u-1\}$, i.e. $d_{r_1}-\hat{\rho}\in\{1,...,1_u\}=\{1,...,d_{r_1}\}$. \\
	\indent	If $1_u>1_x$ then $\hat{\rho}\in\{0,...,1_x-1\}$, i.e. $d_{r_1}-\hat{\rho}\in\{1_u-1_x+1,...,1_u\}=\{d_{r_1}-m_{r_1}+1,...,d_{r_1}\}$.  
	This implies $(d_{r_1}-\hat{\rho})\alpha = m_{r_1}-\hat{\rho}$. 
	We will show that $(d_{r_1}-\hat{\rho})\overline{w}_\alpha = m_{r_1}-\hat{\rho}$. Since $d_{r_1}-\hat{\rho}\in dom(\alpha)$, we can conclude $d_{r_1}-\hat{\rho}\notin A$ by the definition of set $A$.  \\ Suppose that $w_1\in W_u$, i.e. $w_1=w_{s_1}$. Then we can calculate
	$(d_{r_1}-\hat{\rho})\overline{w}_\alpha \\ =(d_{r_1}-\hat{\rho})\overline{v}_A\overline{w}_{s_1}...\overline{w}_{s_a}\overline{w_{s_{a+1}}^{-1}}...\overline{w_{s_{a+b}}^{-1}} \\
	=(d_{r_1}-\hat{\rho})\overline{w}_{s_1}...\overline{w}_{s_a}\overline{w_{s_{a+1}}^{-1}}...\overline{w_{s_{a+b}}^{-1}}$ \quad \quad  \quad \quad \quad  (since  $d_{r_1}-\hat{\rho}\notin A$) \\
	$	=(i_{s_1}-\hat{\rho}+2\lvert{w}_{s_1}...{w}_{s_a}\rvert)\overline{w_{s_{a+1}}^{-1}}...\overline{w_{s_{a+b}}^{-1}}$  \quad \quad \quad    (by Remark \ref{move} and Lemma \ref{tool}(ii))  \\
	$=i_{s_1}-\hat{\rho}+2\lvert{w}_{s_1}...{w}_{s_a}\rvert-2\lvert{w}_{s_{a+1}}^{-1}...{w}_{s_{a+b}}^{-1}\rvert$ \quad \  \ (by Remark \ref{move} and Lemma \ref{tool}(iii))  \\
	$=m_{r_1}-\hat{\rho}$  \quad \quad \quad  \quad \quad  \quad \quad \quad \quad \quad \quad \quad \quad \quad \quad   (by Lemma \ref{x>u}).	\\
	Suppose now $w_1\in W_x$, i.e. $w_1=w_{s_{a+b}}$. We have  $\alpha^{-1}=\bigl(\begin{smallmatrix}
	m_1 & < & \cdots & < & m_p \\
	d_1 &  &  \cdots &    & d_p
	\end{smallmatrix}\bigr)$ and obtain the words
	$w'=w_1'...w'_lw'_{l+1}$ and $w_{\alpha^{-1}}=v_{A'}w_{\alpha^{-1}}^*$ by the same constructions as for the words $w$ and $w_\alpha$, respectively, from $\alpha$. It is easy to verify that $w'_1\in W_u$. We will show that $(d_{r_1}-\hat{\rho})\overline{w}_{\alpha}=m_{r_1}-\hat{\rho}$. As above and using Lemma \ref{w alpha}, we can show that $ d_{r_1}-\hat{\rho}=(m_{r_1}-\hat{\rho})\overline{v}_{A'}\overline{w_{\alpha^{-1}}^*}=(m_{r_1}-\hat{\rho})\overline{v}_{A'}(\overline{w_{\alpha}^*})^{-1}$. Since  $m_{r_1}-\hat{\rho}\in dom (\alpha^{-1})$, i.e. $m_{r_1}-\hat{\rho}\notin A'$, and because $d_{r_1}-\hat{\rho}\notin A$, we obtain 
	$(m_{r_1}-\hat{\rho})(\overline{w_{\alpha}^*})^{-1} = (d_{r_1}-\hat{\rho})\overline{v}_A$. That equation provides $ 
	(m_{r_1}-\hat{\rho})(\overline{w_{\alpha}^*})^{-1}\overline{w_{\alpha}^*} = (d_{r_1}-\hat{\rho})\overline{v}_A\overline{w_{\alpha}^*}$, i.e. $
	m_{r_1}-\hat{\rho} = (d_{r_1}-\hat{\rho})\overline{w}_{\alpha}$.
\end{proof}
Up to this point, we have only considered  such $\rho\in\overline{n}$ which belong to $dom(\alpha)$ and could show that $\rho\alpha = \rho\overline{w}_\alpha$. Now we will consider the remaining elements $\rho$ in $\overline{n}$ and show that $\rho\notin dom(\alpha)$ as well as $\rho\notin dom(\overline{w}_\alpha)$. If we have done it, then we can conclude that both transformations $\alpha$ and $\overline{w}_\alpha$ are equal. \\
\indent First, we consider the intervals  of $\overline{n}$ regarded in Proposition \ref{9}-\ref{begin} and  show that if $\rho\notin dom(\alpha)$ then $\rho\in A$. This provides that $\rho \notin dom(\overline{v}_A)$. Since $dom(\overline{w}_\alpha)= dom(\overline{v}_A\overline{w_\alpha^*})\subseteq dom(\overline{v}_A)$, we obtain $\rho\notin dom(\overline{w}_\alpha)$.
\begin{prop} \label{not}
	Let $k\in\{1,...,a\}$. Then $\{i_{s_k}+2\lvert w_{s_k}\rvert+2,...,(s_k+1)_u-1\}\backslash dom(\alpha)\subseteq A$.
\end{prop}
\begin{proof} 
	By Lemma \ref{A tran= word}, we have $i_{s_k}+2\lvert w_{s_k}\rvert+2=d_{r_{{s_k}+1}}-(m_{r_{s_k+1}}-m_{r_{s_k}}-2)$, where $({s_k}+1)_u-1=d_{r_{{s_k}+1}}-1$ by Lemma \ref{lem zu=dz}. Then $\{i_{s_k}+2\lvert w_{s_k}\rvert+2,...,(s_k+1)_u-1\}=\{d_{r_{{s_k}+1}}-(m_{r_{{s_k}+1}}-m_{r_{s_k}}-2),...,d_{r_{{s_k}+1}}-1\}$. Let $\rho\in\{d_{r_{{s_k}+1}}-(m_{r_{{s_k}+1}}-m_{r_{s_k}}-2),...,d_{r_{{s_k}+1}}-1\}$.
	If $\rho\leq d_{r_{s_k}+1}-1$ then we obtain $\rho\in A$ by (3.2) of the definition of set $A$. 
	If $\rho>d_{r_{s_k}+1}$ then there is $t\in\{r_{s_k}+1,...,r_{{s_k}+1}-1\}$ such that $\rho\in\{d_{t}+1,...,d_{t+1}-1\}$. Then $\rho\in A$ by (3.1) of the definition of set $A$.
\end{proof} 
\begin{prop} \label{not1} Let $d\in\{1,...,b\}$. Then $\{(s_{a+d}+1)_u-b_{s_{a+d}},...,(s_{a+d}+1)_u-1\}\backslash dom(\alpha)\subseteq A.$ 
\end{prop}
\begin{proof} 	We have $({s_{a+d}}+1)_u -1 = d_{r_{{s_{a+d}}+1}}-1$ and $({s_{a+d}}+1)_u-b_{s_{a+d}}=d_{r_{s_{a+d}}}+2$ by Lemmas \ref{lem zu=dz} and  \ref{8}, respectively.   Then $\{({s_{a+d}}+1)_u-b_{s_{a+d}},...,({s_{a+d}}+1)_u-1\}=\{d_{r_{s_{a+d}}}+2,...,d_{r_{{s_{a+d}}+1}}-1\}$. Let $\rho\in\{d_{r_{s_{a+d}}}+2,...,d_{r_{{s_{a+d}}+1}}-1\}$. 
	If $\rho\leq d_{r_{s_{a+d}}+1}-1$ then we obtain $\rho\in A$ by  (3.3) of the definition of set $A$. 
	If $\rho>d_{r_{s_{a+d}}+1}$ then there is $t\in\{r_{s_{a+d}}+1,...,r_{{s_{a+d}}+1}-1\}$ such that $\rho\in\{d_{t}+1,...,d_{t+1}-1\}$. Then $\rho\in A$ by  (3.1) of the definition of set $A$.
\end{proof}
Finally, we consider the case that $m_p\neq d_p$. In this case, we have $w_{l+1}\in W_u\cup W_x$.
\begin{prop} \label{not2}
	If $m_p>d_p$ and  $ i_{s_a}+2\lvert w_{s_a}\rvert+2\leq n $  then $\{i_{s_a}+2\lvert w_{s_a}\rvert+2,...,n\}\subseteq A$.
\end{prop}
\begin{proof} We have $l+1=s_a$, $r_{l+1}=p$, $w_{l+1}\in W_u$, and $(l+1)_x=i_{l+1}+2j_{l+1}$. Then 
	$i_{l+1}+2j_{l+1}+2 = (l+1)_x+2
	= m_{r_{l+1}}+2=m_p+2$ by Lemma \ref{lem zu=dz}.
	Thus, $\{i_{s_a}+2\lvert w_{s_a}\rvert+2,...,n\}=\{i_{l+1}+2\lvert w_{l+1}\rvert+2,...,n\}=\{m_p+2,...,n\}$. Then by  (1.2) of the definition of set $A$, we get that $\{i_{s_a}+2\lvert w_{s_a}\rvert+2,...,n\}\subseteq A$.
\end{proof}
Using (1.1) from the definition of set $A$ instead of (1.2), we obtain similarly:
\begin{prop} \label{short proof}
	If $m_p<d_p$  and $ i_{s_{a+1}}+2\lvert w_{s_{a+1}}^{-1}\rvert+2\leq n $ then $\{i_{s_{a+1}}+2\lvert w_{s_{a+1}}^{-1}\rvert+2,...,n\}\subseteq A$.
\end{prop}

Now, we consider the remaining intervals in $\overline{n}$. We start with the interval before $d_1$ and after $d_p$, respectively. \\

\indent By (1.3), (2), (5), (6), and (3.1), respectively, of the definition of set $A$, we obtain immediately: 
\begin{prop} \label{21} 
	(i)	If $m_p=d_p<n$ then $\{m_p+1,...,n\}\subseteq A$. \\
	(ii)	If $1<d_1\leq m_1$ then $\{1,...,d_1-1\}\subseteq A$.  \\
	(iii)	If $1<m_1< d_1$ then $\{d_1-m_1+1,...,d_1-1\}\subseteq A$.  \\
	(iv)	If $1_u\neq 1_x$ and $1\notin \{1_u,1_x\}$ then $\{d_{t}+1,...,d_{t+1}-1\}\subseteq A$ for all $t\in\{1,...,r_1-1\}$.   
\end{prop}
\begin{prop} \label{111}
	If $1_x<1_u$  then $\rho\notin dom(\overline{w}_\alpha)$ for all  $\rho\in\{1,...,1_u-1_x\}$.
\end{prop}
\begin{proof}
	If $w_1\in W_u$ then $w_1=w_{s_1}$. By Lemmas \ref{lem zu=dz} and \ref{x>u}, we have that $1_x=i_{s_1}+2\lvert w_{s_1}...w_{s_a}\rvert-2\lvert w^{-1}_{s_{a+b}}...w^{-1}_{s_{a+1}}\rvert$ and $1_u=i_{s_1}$. 	If $w_1\in W_x$ then $w_1=w_{s_{a+b}}$. By Lemmas \ref{lem zu=dz} and \ref{u>x}, we have that $1_x=i_{s_{a+b}}$ and $1_u=i_{s_{a+b}}+2\lvert w^{-1}_{s_{a+b}}...w^{-1}_{s_{a+1}}\rvert-2\lvert w_{s_1}...w_{s_a}\rvert$.\\
	\indent Then $1_u-1_x= 2\lvert w^{-1}_{s_{a+b}}...w^{-1}_{s_{a+1}}\rvert-2\lvert w_{s_1}...w_{s_a}\rvert=2k$ for some positive integer $k$. We put $\mathcal{U}=w_{s_1}...w_{s_a}$ and $\mathcal{X}=w^{-1}_{s_{a+b}}...w^{-1}_{s_{a+1}}$, i.e. $2k=2\lvert\mathcal{X}\rvert-2\lvert\mathcal{U}\rvert$ and $\lvert\mathcal{X}\rvert=\lvert\mathcal{U}\rvert+k$. Let $\overline{w_{s_{a+1}}^{-1}}...\overline{w_{s_{a+b}}^{-1}}=\overline{y}_1...\overline{y}_{\lvert\mathcal{U}\rvert}\overline{y}_{\lvert\mathcal{U}\rvert+1}...\overline{y}_{\lvert\mathcal{U}\rvert+k}$, where $y_1,...,y_{\lvert\mathcal{U}\rvert+k}\in\{x_1,...,x_{n-2}\}$. Let  $\rho\in\{1,...,1_u-1_x\}$. Clearly, $\rho\notin dom(\alpha)$ and $\rho\notin A$   by the definition of  set $A$. On the other hand, we have 
	$	\rho\overline{v}_A\overline{w_\alpha ^*}\\
	=\rho\overline{w}_{s_1}...\overline{w}_{s_a}\overline{y}_1...\overline{y}_{\lvert\mathcal{U}\rvert}\overline{y}_{\lvert\mathcal{U}\rvert+1}...\overline{y}_{\lvert\mathcal{U}\rvert+k} \ \quad \quad \indent (\mbox{since} \ \rho\notin A) \\
	= (\rho+2\lvert w_{s_1}...w_{s_a}\rvert)\overline{y}_1...\overline{y}_{\lvert\mathcal{U}\rvert}\overline{y}_{\lvert\mathcal{U}\rvert+1}...\overline{y}_{\lvert\mathcal{U}\rvert+k}$ \ (since  $\rho<1_u\leq i_{s_1}$  and by Remark \ref{move}  and Lemma \ref{tool}(ii)). \\
	Using Lemma \ref{1-4}(i), it is routine to calculate that $2\lvert w^{-1}_{s_{a+b}}...w^{-1}_{s_{a+1}}\rvert< i_{s_{a+1}}+2\lvert w^{-1}_{s_{a+1}}\rvert$, i.e.
	$(1_u-1_x)+2\lvert w_{s_1}...w_{s_a}\rvert=   2\lvert w^{-1}_{s_{a+b}}...w^{-1}_{s_{a+1}}\rvert< i_{s_{a+1}}+2\lvert w^{-1}_{s_{a+1}}\rvert$. This implies $\rho+2\lvert w_{s_1}...w_{s_a}\rvert\leq i_{s_{a+1}}+2\lvert w^{-1}_{s_{a+1}}\rvert$. Then 
	$(\rho+2\lvert w_{s_1}...w_{s_a}\rvert)\overline{y}_1...\overline{y}_{\lvert\mathcal{U}\rvert}\overline{y}_{\lvert\mathcal{U}\rvert+1}...\overline{y}_{\lvert \mathcal{U}\rvert+k}=\rho\overline{y}_{\lvert\mathcal{U}\rvert+1}...\overline{y}_{\lvert\mathcal{U}\rvert+k}$ using Remark \ref{move}. Note that $1_u-1_x$ is even and there is $i\in\{2,4,...,1_u-1_x\}$   such that 
	$\rho\in\{i-1,i\}$. 
	If $\rho=i-1$ then $\rho-2\lvert{y}_{\lvert\mathcal{U}\rvert+1}...{y}_{\lvert\mathcal{U}\rvert+\frac{i}{2}-1}\rvert=1$.
	If $\rho=i$ then $\rho-2\lvert{y}_{\lvert\mathcal{U}\rvert+1}...{y}_{\lvert\mathcal{U}\rvert+\frac{i}{2}-1}\rvert=2$. Therefore, $	\rho\overline{y}_{\lvert\mathcal{U}\rvert+1}...\overline{y}_{\lvert\mathcal{U}\rvert+\frac{i}{2}}...\overline{y}_{\lvert\mathcal{U}\rvert+\frac{1_u-1_x}{2}} = (\rho-2\lvert{y}_{\lvert\mathcal{U}\rvert+1}...{y}_{\lvert\mathcal{U}\rvert+\frac{i}{2}-1}\rvert)\overline{y}_{\lvert\mathcal{U}\rvert+\frac{i}{2}}...\overline{y}_{\lvert\mathcal{U}\rvert+\frac{1_u-1_x}{2}}$  by Remark \ref{move}, i.e. 
	$\rho\overline{y}_{\lvert\mathcal{U}\rvert+1}...\overline{y}_{\lvert\mathcal{U}\rvert+\frac{i}{2}}...\overline{y}_{\lvert\mathcal{U}\rvert+\frac{1_u-1_x}{2}}	= \hat{\rho}\overline{y}_{\lvert\mathcal{U}\rvert+\frac{i}{2}}...\overline{y}_{\lvert\mathcal{U}\rvert+\frac{1_u-1_x}{2}}, \mbox{where} \ \hat{\rho}\in\{1,2\}$.
	Then we have $\hat{\rho}\notin dom(\overline{y}_{\lvert\mathcal{U}\rvert+\frac{i}{2}})$. This implies $\rho\notin dom(\overline{w}_\alpha)$.
\end{proof}  
It is easy to check that any interval $I$ of $\overline{n}$, which we have not yet regarded, belongs to an interval of the form $\{d_i+1,...,d_{i+1}-1\}$ for some $i\in\{1,....,l\}$, i.e. $I\cap dom(\alpha)=\emptyset$. It remains to show that $I\cap dom(\overline{w}_\alpha)=\emptyset$.
\begin{prop} \label{notin}
	Let $k\in\{1,...,a\}$ and let $\rho\in\{i_{s_k}+1,...,i_{s_k}+2\lvert w_{s_k}\rvert+1\}\cap \overline{n}$. Then $\rho\notin dom(\overline{w}_\alpha)$.
\end{prop}
\begin{proof} 
	First, we have $
	\rho\overline{v}_A\overline{w_\alpha^*} = \rho\overline{w}_{s_1}...\overline{w}_{s_a}\overline{w_{s_{a+1}}^{-1}}...\overline{w_{s_{a+b}}^{-1}}$ since  $\rho\notin A$ by the  definition of set   $A$ and 
	$\rho\overline{w}_{s_1}...\overline{w}_{s_a}\overline{w_{s_{a+1}}^{-1}}...\overline{w_{s_{a+b}}^{-1}}=\rho\overline{w}_{s_k}...\overline{w}_{s_a}\overline{w_{s_{a+1}}^{-1}}...\overline{w_{s_{a+b}}^{-1}}$ by Remark \ref{move} and Lemma \ref{tool}(i), where $w_{s_k}=u_{i_{s_k}}u_{i_{s_k}+2}...u_{i_{s_k}+2\lvert w_{{s_k}}\rvert-2}$.
	If $\rho\in\{i_{s_k}+1,i_{s_k}+2,i_{s_k}+3\}\cap\overline{n}$ then $\rho\notin dom (\overline{u}_{i_{s_k}})$ by Remark \ref{move}. 
	If $\rho=i_{s_k}+h+t$ for some $h\in\{2,4,...,2\lvert w_{s_k}\rvert-2\},  t\in\{2,3\}$ then \\
	$\rho\overline{u}_{i_{s_k}}\overline{u}_{i_{s_k}+2}...\overline{u}_{i_{s_k}+2\lvert w_{s_k}\rvert-2}\overline{w}_{s_{k+1}}...\overline{w}_{s_a}\overline{w_{s_{a+1}}^{-1}}...\overline{w_{s_{a+b}}^{-1}} \\=(i_{s_k}+h+t)\overline{u}_{i_{s_k}+h}...\overline{u}_{i_{s_k}+2\lvert w_{s_k}\rvert-2}\overline{w}_{s_{k+1}}... \overline{w}_{s_a}\overline{w_{s_{a+1}}^{-1}}...\overline{w_{s_{a+b}}^{-1}}$ \quad by Remark \ref{move}.  \\ We observe that $i_{s_k}+h+t\notin dom (\overline{u}_{i_{s_k}+h})$. Hence, $\rho \notin dom (\overline{w}_\alpha)$.
\end{proof}
\begin{prop}    \label{<n}
	Let $d\in\{1,...,b\}$ with $(s_{a+d})_u<n$. Then $(s_{a+d})_u+1\notin dom(\overline{w}_\alpha)$.
\end{prop}
\begin{proof}
	Assume $(s_{a+d})_u+1\in dom(\overline{w}_\alpha)$. We have shown that  $(s_{a+d})_u\overline{w}_\alpha = (s_{a+d})_u\alpha = (s_{a+d})_x$ in Proposition \ref{12}. Recall that $\overline{w}_\alpha\in IOF_n^{par}$. Then Proposition \ref{4 choice}(i, iii) implies  
	$((s_{a+d})_u+1)\overline{w}_\alpha=(s_{a+d})_x+1$  and $ 
	(s_{a+d})_u+1 = ((s_{a+d})_u+1)\overline{w}_\alpha(\overline{w}_\alpha)^{-1} = ((s_{a+d})_x+1)(\overline{w}_\alpha)^{-1}$,
	i.e. $(s_{a+d})_x+1\in dom((\overline{w}_\alpha)^{-1})$. We have \\
	 $
	(\overline{w}_\alpha)^{-1}=(\overline{v}_A\overline{{w}}_{s_1}...\overline{{w}}_{s_a}{\overline{w_{s_{a+1}}^{-1}}}...{\overline{w_{s_{a+b}}^{-1}}})^{-1}  \\ 
	= ({\overline{x^{-1}_{i_{s_{a+b}},j_{s_{a+b}}}}})^{-1}...({\overline{x^{-1}_{i_{s_{a+1}},j_{s_{a+1}}}}})^{-1}(\overline{u}_{i_{s_{a}},j_{s_{a}}})^{-1} ...(\overline{u}_{i_{s_{1}},j_{s_{1}}})^{-1}(\overline{v}_A)^{-1} \\  = \overline{u}_{i_{s_{a+b}},j_{s_{a+b}}}...\overline{u}_{i_{s_{a+1}},j_{s_{a+1}}}\overline{x^{-1}_{i_{s_{a}},j_{s_{a}}}}...\overline{x^{-1}_{i_{s_{1}},j_{s_{1}}}} (\overline{v}_A)^{-1}$. \\
	Since $w_{s_{a+d}}\in W_x$, we get  $(s_{a+d})_x=i_{s_{a+d}}$.  Then   $ (i_{s_{a+d}}+1)(\overline{w}_\alpha)^{-1} \\ =(i_{s_{a+d}}+1)\overline{u}_{i_{s_{a+b}},j_{s_{a+b}}}... \overline{u}_{i_{s_{a+d}},j_{s_{a+d}}}...\overline{u}_{i_{s_{a+1}},j_{s_{a+1}}}  \overline{x^{-1}_{i_{s_{a}},j_{s_{a}}}}...\overline{x^{-1}_{i_{s_{1}},j_{s_{1}}}} (\overline{v}_A)^{-1} \\ = (i_{s_{a+d}}+1)\overline{u}_{i_{s_{a+d}},j_{s_{a+d}}}...\overline{u}_{i_{s_{a+1}},j_{s_{a+1}}}\overline{x^{-1}_{i_{s_{a}},j_{s_{a}}}}...\overline{x^{-1}_{i_{s_{1}},j_{s_{1}}}} (\overline{v}_A)^{-1}$ by Remark \ref{move} and Lemma \ref{tool}(i). \\
	Clearly, $i_{s_{a+d}}+1\notin dom (\overline{u}_{i_{s_{a+d}}})$ and thus, $(s_{a+d})_x+1\notin dom ((\overline{w}_\alpha)^{-1})$, a contradiction.
\end{proof}
Now, we can summarize all results and obtain that $\alpha$ and $\overline{w}_\alpha$ are equal.
\begin{thm} \label{alpha}
	$\alpha$ = $\overline{w}_\alpha$.
\end{thm}
\begin{proof}
	Let $\rho\in\{1,...,n\}$. Then $\rho\in\{1,...,d_{r_1}\}$ or
	$\rho\in\{d_{r_{k-1}}+1,...,d_{r_k}\}$ for some $k\in\{2,...,a+b\}$  or $\rho\in\{d_{r_{a+b}}+1,...,n\}$. We have to show that $\rho\alpha= \rho\overline{w}_\alpha$,  whenever $\rho\in dom(\alpha)$ and $\rho\notin dom (\overline{w}_\alpha)$, whenever $\rho\notin dom(\alpha)$. If $\rho\in\{1,...,d_{r_1}\}$  we can conclude it by     Propositions \ref{begin}, \ref{21}(ii-iv), and \ref{111}. If 	$\rho\in\{d_{r_{k-1}}+1,...,d_{r_k}\}$ then we can conclude it by Propositions  \ref{9}, \ref{10}, \ref{12}, \ref{not}, \ref{not1}, \ref{notin}, and \ref{<n}. If $\rho\in\{d_{r_{a+b}}+1,...,n\}$  then we can conclude it by Propositions   \ref{mp=dp}, \ref{mp'}, \ref{not2}, \ref{short proof}, \ref{21}(i), \ref{notin}, and \ref{<n}.
\end{proof}
\begin{cor} \label{An}
	$IOF_n^{par} = \langle A_n \rangle$.
\end{cor}
\begin{proof} We have already mentioned that $\langle A_n \rangle\subseteq IOF_n^{par}$.
	Theorem \ref{alpha} shows $\alpha$ = $\overline{w}_\alpha$, where $\overline{w}_\alpha\in\langle A_n, \overline{u}_{n-3}, \overline{x}_{n-3} \rangle$. It is easy to verify that $\overline{u}_{n-3}=\overline{v}_{n-2}\overline{u}_{n-2}$ and $\overline{x}_{n-3}=\overline{v}_{n}\overline{x}_{n-2}$, where $\overline{u}_{n-2}, \overline{x}_{n-2}, \overline{v}_{n-2}, \overline{v}_{n}\in A_n$. Hence, $\alpha\in \langle A_n \rangle$. Since we have proved $\alpha$ = $\overline{w}_\alpha$ for any $\alpha\in IOF_n^{par}$, we can conclude that $IOF_n^{par} \subseteq \langle A_n \rangle$, which completes the proof.
\end{proof}
\section{The rank of $IOF_n^{par}$}
In this section, we provide the main result of that paper, the rank of $IOF_n^{par}$.
We can calculate that $\lvert A_n\rvert=2(n-3)+n=3n-6$. Moreover, $A_n$ is a generating set of the monoid $IOF_n^{par}$ by Corollary \ref{An}.  This provides: 
\begin{prop} \label{small}
	rank$(IOF_n^{par})\leq 3n-6$.  
\end{prop}   
We have still to show that  rank$(IOF_n^{par})\geq 3n-6$.  First, we consider the transformations with rank $n-1$. Clearly, for any transformation $\alpha$ with rank $n-1$, there is $i\in\{1,...,n\}$ such that $dom(\alpha) = dom(\overline{v}_i)$.  As an immediate consequence of Proposition \ref{4 choice}(i, ii), we obtain:
\begin{lem} \label{preserving} Let $\alpha\in IOF_n^{par}$ and $i\in\{1,...,n\}$ with $dom(\alpha)=dom(\overline{v}_i)$. Then $\alpha=\overline{v}_i$. 
\end{lem}
\begin{lem} \label{partial id} Let $G$ be a generating set of $IOF_n^{par}$. Then $\overline{v}_1,...,\overline{v}_n\in G$. 
\end{lem}   
\begin{proof}
	Let $i\in\{1,...,n\}$. Then there exist $\alpha_1,...,\alpha_m \in G\backslash\{ id_{\overline{n}}\}$ such that $\overline{v}_i=\alpha_1\cdot\cdot\cdot\alpha_m$, where $dom(\overline{v}_i)\subseteq dom(\alpha_1)$. Since $\alpha_1\neq id_{\overline{n}} $, we have rank$(\alpha_1)=n-1$, i.e. $dom(\overline{v}_i)= dom(\alpha_1)$. By Lemma \ref{preserving}, we get that $\overline{v}_i=\alpha_1$.
\end{proof}
\begin{lem} \label{...} Let  $\alpha\in IOF_n^{par}$ with rank$(\alpha)=n-2$. Then $\alpha=\overline{v}_A$ with  $A=\overline{n}\backslash dom(\alpha)$  or $\alpha=\overline{u}_{n-2}$ or $\alpha=\overline{x}_{n-2}$.
\end{lem}
\begin{proof}
	Let $\alpha = \bigl(\begin{smallmatrix}
	d_1 & < & \cdots & < & d_{n-2} \\
	m_1 &  &  \cdots &    & m_{n-2}
	\end{smallmatrix}\bigr)$.
	Assume there is  $i\in\{2,...,n-2\}$ such that $d_i-d_{i-1} \neq m_i-m_{i-1}$, i.e.  $d_i-d_{i-1}>m_i-m_{i-1}$ or $d_i-d_{i-1}<m_i-m_{i-1}$. 
	If $d_i-d_{i-1}>m_i-m_{i-1}$ then $m_i-m_{i-1}>1$ and thus, $d_i-d_{i-1}\geq 4$, by Proposition \ref{4 choice}(iii, iv). This implies $\lvert dom(\alpha)\rvert< n-3, $ a contradiction. 
	If $d_i-d_{i-1}<m_i-m_{i-1}$ then we obtain $\lvert im(\alpha)\rvert<n-3$ by dually arguments, a contradiction. Therefore,  $d_i-d_{i-1}=m_i-m_{i-1}$ for all $i\in\{2,...,n-2\}$, which together with Proposition \ref{4 choice}(ii) implies $k\alpha=k$ for all $k\in dom(\alpha)$ or $k\alpha=k+2$ for all $k\in dom(\alpha)$ or $k\alpha=k-2$ for all $k\in dom(\alpha)$. Hence, $\alpha=\overline{v}_A$  with  $A=\overline{n}\backslash dom(\alpha)$ or $\alpha=\overline{u}_{n-2}$ or $\alpha=\overline{x}_{n-2}$.
\end{proof}
\begin{lem} \label{n-2}
	Let $G$ be a generating set of $IOF_n^{par}$. Then $\overline{u}_{n-2},\overline{x}_{n-2}\in G$.
\end{lem}
\begin{proof} First, we show that $\overline{u}_{n-2}\in G$.  There are $\alpha_1,...,\alpha_m \in G\backslash \{id_{\overline{n}}\}$ such that  $\overline{u}_{n-2}=\alpha_1\cdot\cdot\cdot\alpha_m$. Then $dom(\overline{u}_{n-2})\subseteq dom(\alpha_1)$. If $dom(\overline{u}_{n-2})\subset dom(\alpha_1)$ then we get $\alpha_1\in {Id_{\overline{n}}}$ by Lemma \ref{preserving}. If $dom(\alpha_1)=dom(\overline{u}_{n-2})\neq dom(\overline{x}_{n-2}) $ then $\alpha_1\in  Id_{\overline{n}}$ or $\alpha_1=\overline{u}_{n-2}$ by Lemma \ref{...}. 	If $\alpha_1\in Id_{\overline{n}}$ then $dom(\overline{u}_{n-2})\subseteq dom(\alpha_2)$, i.e. $\overline{u}_{n-2}= \alpha_2$    or $\alpha_2 \in Id_{\overline{n}}$ by the same arguments like for $\alpha_1$. Continuing 
	that procedure, we obtain that either there is $j\in\{1,...,m\}$ such that $\overline{u}_{n-2}= \alpha_j$    or $\alpha_1,...,\alpha_m\in Id_{\overline{n}}$. Assume that $\alpha_1,...,\alpha_m\in Id_{\overline{n}}$. Then $\overline{u}_{n-2}=\alpha_1\cdot\cdot\cdot\alpha_m\in Id_{\overline{n}}$ since $Id_{\overline{n}}$ is a submonoid of $I_n$, a contradiction. Hence, there exists $j\in\{1,...,m\}$ such that $\overline{u}_{n-2}= \alpha_j\in G$. Dually, we can show $\overline{x}_{n-2}\in G$.
\end{proof}
For $i\in\{1,...,n-4\}$, we put $J_i=\{1,...,i,i+4,...,n\}$.
\begin{lem} \label{n-4}
	Let $G$ be a generating set of $IOF_n^{par}$. Then there are pairwise different $\beta_1,...,\beta_{n-4}, \\ \gamma_1,...,\gamma_{n-4}\in G$ such that  $dom(\gamma_i)=J_i=im(\beta_i)$ for all $i\in\{1,...,n-4\}$.
\end{lem}
\begin{proof} Note, for $n=4$, we observe that the statement of this lemma is true trivially. We are going to show the rest of the proof for $n\geq 5$. Let $i\in\{1,...,n-4\}$. It is easy to verify that there is $\alpha \in IOF_n^{par}\backslash Id_{\overline{n}}$ such that $dom(\alpha)=J_i$. Then there are  $\alpha_1,...,\alpha_m\in G\backslash \{id_{\overline{n}}\}$ such that  $\alpha=\alpha_1\cdot\cdot\cdot\alpha_m$. In particular, we have $dom(\alpha)\subseteq dom(\alpha_1)$. If  rank$(\alpha)$ < rank$(\alpha_1)$ then $\alpha_1\in Id_{\overline{n}}$ or $\alpha_1\in\{\overline{u}_{n-2}, \overline{x}_{n-2}\}$ by Lemma \ref{...}. Since $dom(\alpha)\subseteq dom(\alpha_1)$, we get $\alpha_1\notin\{\overline{u}_{n-2}, \overline{x}_{n-2}\}$, i.e. $\alpha_1\in Id_{\overline{n}}$. 
	If rank$(\alpha)$ = rank$(\alpha_1)$ then $dom(\alpha)=dom(\alpha_1)$. Suppose $\alpha_1\in Id_{\overline{n}}$. Then $dom(\alpha)\subseteq dom(\alpha_2)$. 
	If rank$(\alpha)$ < rank$(\alpha_2)$ then we obtain $\alpha_2\in Id_{\overline{n}}$ by the same argument like for $\alpha_1$. 
	If rank$(\alpha)$ = rank$(\alpha_2)$ then $dom(\alpha)=dom(\alpha_2)$. Continuing that procedure, we obtain  that either $\alpha_1,...,\alpha_m \in Id_{\overline{n}}$ or  there is $j\in\{1,...,m\}$ such that $dom(\alpha)=dom(\alpha_j)$ and $\alpha_j\notin Id_{\overline{n}}$. Note $\alpha_1,...,\alpha_m \in Id_{\overline{n}}$ is not possible since $\alpha_1\cdot\cdot\cdot\alpha_m=\alpha\notin Id_{\overline{n}}$. We put $\gamma_i=\alpha_j$ and we have $dom(\gamma_i)=dom(\alpha)$ and $\gamma_i\notin Id_{\overline{n}}$. \\
	\indent Next, we show that there is $\beta_i\in G\backslash Id_{\overline{n}}$ with $im(\beta_i)=J_i$.
	It is easy to verify that there is $\alpha \in IOF_n^{par}\backslash Id_{\overline{n}}$, with $im(\alpha)=J_i$.  Then there are $\alpha_1,...,\alpha_m\in G$ such that $\alpha=\alpha_1\cdot\cdot\cdot\alpha_m$. We have that $im(\alpha)\subseteq im(\alpha_m)$. If $im(\alpha)\subset im(\alpha_m) $ then $\alpha_m\in Id_{\overline{n}}$  or $\alpha_m\in\{\overline{u}_{n-2},\overline{x}_{n-2}\}$ by Lemma \ref{...}. Since $im(\alpha)\subseteq im(\alpha_m) $,  we can conclude that $\alpha_m\notin\{\overline{u}_{n-2},\overline{x}_{n-2}\}$. Consequently, either $im(\alpha)= im(\alpha_m) $ and $\alpha_m\notin Id_{\overline{n}}$ or $\alpha_m\in Id_{\overline{n}}$.	Suppose $\alpha_m\in Id_{\overline{n}}$. Then  $im(\alpha)\subseteq im(\alpha_{m-1})$. By the same argument as for $\alpha_m$, we obtain that either $im(\alpha) = im(\alpha_{m-1})$ and $\alpha_{m-1}\notin Id_{\overline{n}}$ or $\alpha_{m-1}\in Id_{\overline{n}}$. Continuing that procedure, we obtain that either $\alpha_1,...,\alpha_m\in Id_{\overline{n}}$ or there is $j\in\{1,...,m\}$ such that $im(\alpha_j)=im(\alpha)$ and  $\alpha_j\notin Id_{\overline{n}}$. The case $\alpha_1,...,\alpha_m\in Id_{\overline{n}}$ is not possible since $\alpha_1\cdot\cdot\cdot\alpha_m=\alpha\notin Id_{\overline{n}}$. We put $\beta_i=\alpha_j$ and we have $im(\beta_i)=im(\alpha)$ and $\beta_i\notin Id_{\overline{n}}$. \\	
	\indent Let now $i,j \in\{1,...,n-4\}$. Assume that $\gamma_i=\beta_j$. Then $J_i=dom(\gamma_i)=dom(\beta_j)$ and $im(\gamma_i)=im(\beta_j)=J_j$. It is easy to see by Proposition \ref{4 choice} that $k\gamma_i=k$ for all $k\in\{1,...,i,i+4,...,n\}$. Hence, we have $\gamma_i\in Id_{\overline{n}}$, a contradiction.
\end{proof}
Lemma \ref{partial id} provides $n$ transformations of rank $n-1$ which have to belong to any generating set of $IOF_n^{par}$. Lemma \ref{n-2} provides two transformations of rank $n-2$. Finally, Lemma \ref{n-4} points out that any generating set of $IOF_n^{par}$ has to contain $2(n-4)$ transformations of rank $n-3$. This shows that any generating set of $IOF_n^{par}$ contains at least $n+2+2(n-4) = 3n-6$ transformations, which proves:
\begin{prop} \label{big} rank$(IOF_n^{par})\geq 3n-6$.
\end{prop}
By Proposition \ref{small} and \ref{big}, we can state the main result:
\begin{thm} 
	rank$(IOF_n^{par})=3n-6$.
\end{thm}

\section*{Statements and Declarations}
There are no financial or non-financial interests directly or indirectly related
to the work submitted for publication.






\bibliography{sn-bibliography}

\end{document}